%% file: GDRO0427.tex
\documentclass[final,onefignum,onetabnum]{siamart220329}
\pdfoutput=1
\input{GDRO_shared}

\ifpdf
\hypersetup{
  pdftitle={Globalized distributionally robust optimization based on samples},
  pdfauthor={Yueyao Li, and Wenxun Xing}
}
\fi

\begin{document}

\maketitle
% REQUIRED
\begin{abstract}
It is known that the set of perturbed data is key in robust optimization modelling. Distributionally robust optimization (DRO) is a methodology used for optimization problems affected by random parameters with uncertain probability distribution. In terms of the information of the perturbed data, it is essential to estimate an appropriate support set of the probability distribution in formulating DRO models. In this paper, we introduce two globalized distributionally robust optimization (GDRO) models which choose a core set based on data and a sample space containing the core set to balance the degree of robustness and conservatism at the same time. The degree of conservatism can be controlled by the expected distance of random parameters from the core set. Under some assumptions, we further reformulate several GDRO models into tractable semi-definite programs. In addition, numerical experiments are provided showing the relationship between the optimal objective values of the GDRO models and the size of the sample space and the core set.
\end{abstract}

% REQUIRED

% REQUIRED

\section{Introduction}\label{intro}
Distributionally robust optimization (DRO) model is generally formulated as
\begin{equation}
\label{DRO}
\min \limits_{\bm{x} \in X} \max \limits_{P \in \mathcal{D}} E_{P}[h(\bm{x},\bm{\xi})],
\end{equation}
where $\bm{x} \in \mathbb{R}^n$ is the decision vector, $X \subseteq \mathbb{R}^n$ is a closed convex set of feasible solutions, $h(\bm{x},\bm{\xi})$ is a convex function in $\bm{x}$ that depends on the random vector $\bm{\xi} \in \mathbb{R}^p$, $E_P[\cdot]$ is the expectation taken with respect to $\bm{\xi}$ and $\mathcal{D}$ is the ambiguity set. The only information we know about $\bm{\xi}$ is a set of samples of it.

There are wide range of ambiguity sets constructed by different methods, such as matrix moment constraints \cite{matrix}, Wasserstein metric \cite{wasserstein} and polynomial density functions \cite{polynomial}. Here we consider the ambiguity set of matrix moment constraints \cite{ye} as follows,
\begin{equation}\label{ambiguity}
   \mathcal{D}=
   \left\{
   P \in \mathcal{P}
   \middle\arrowvert
   \begin{array}{lcl}
    & (E_{P}[\bm{\xi}]-\bm{\mu}_0)^{T}\bm{\Sigma}_0^{-1}(E_{P}[\bm{\xi}]-\bm{\mu}_0) \leq \gamma_1  \\
    & E_{P}[(\bm{\xi}-\bm{\mu}_0)(\bm{\xi}-\bm{\mu}_0)^{T}] \preceq \gamma_2\bm{\Sigma}_0\\
    & P(\bm{\xi} \in \Xi) = 1
   \end{array}
   \right\},
\end{equation}
where $\mathcal{P}$ is the set of all probability measures on the measurable space $(\Xi,\mathcal{B}(\Xi))$, $\Xi \subseteq \mathbb{R}^p$ is the sample space, $\mathcal{B}(\Xi)$ is the Borel $\sigma$-algebra on $\Xi$, $\bm{\xi}: \Xi \to \mathbb{R}^p$ is a vector of random variables in probability space $(\Xi,\mathcal{B}(\Xi),P)$ satisfying $\bm{\xi}(\bm{\omega})=\bm{\omega}$,
$\bm{\mu}_0$ is the sample mean, $\bm{\Sigma}_0$ is the sample covariance matrix and $\gamma_1$ and $\gamma_2$ are estimated by statistical method to guarantee that the true mean and covariance matrix are contained with high probability. Details of determining $\gamma_1$ and $\gamma_2$ can be seen in Appendix B.

In this paper, we present a new globalized distributionally robust optimization (GDRO) framework, in which a core set based on given samples is fixed and the sample space of probability distribution is globalized. GDRO can be regarded as an extension of the traditional DRO. Here two kinds of GDRO models are provided. The first one called GDRO1 embodying the globalization in the objective function in \cref{DRO} is
\begin{equation}
\label{GDRO1}
\min \limits_{\bm{x} \in X} \max \limits_{P \in \mathcal{D}} E_{P}[h(\bm{x},\bm{\xi})-\min \limits_{\bm{\xi}' \in Y} \phi(\bm{\xi},\bm{\xi}')],
\end{equation}
where the core set $Y$ is a convex compact set containing core samples and $\phi:\mathbb{R}^p \times \mathbb{R}^p \to \mathbb{R}$ is a closed, jointly convex and nonnegative function for which $\phi(\bm{\xi},\bm{\xi})=0$ for all $\bm{\xi} \in \mathbb{R}^p$. The ambiguity set $\mathcal{D}$ remains to be \cref{ambiguity}.

Besides, the set $Y$  and $\Xi$ are defined by
\begin{equation*}
\begin{split}
Y = \{ \bm{\mu}_0 + \bm{A}\bm{\zeta} | \bm{\zeta} \in Z_1 \},  \\
\Xi = \{ \bm{\mu}_0 + \bm{A}\bm{\zeta} | \bm{\zeta} \in Z_2 \},
\end{split}
\end{equation*}
where $Z_1$ is a given nonempty, convex and compact set with $0 \in ri(Z_1)$ and $Z_2$ is also convex and contains $Z_1$.

We obverse that
\begin{equation*}
\begin{split}
& E_{P}[h(\bm{x},\bm{\xi})-\min \limits_{\bm{\xi}' \in Y} \phi(\bm{\xi},\bm{\xi}')] \\
= & \int_{\Xi} (h(\bm{x},\bm{\xi})-\min \limits_{\bm{\xi}' \in Y} \phi(\bm{\xi},\bm{\xi}')) dP \\
= & \int_{Y} h(\bm{x},\bm{\xi}) dP + \int_{\Xi \backslash Y} (h(\bm{x},\bm{\xi})-\min \limits_{\bm{\xi}' \in Y} \phi(\bm{\xi},\bm{\xi}')) dP,
\end{split}
\end{equation*}
where $\min_{\bm{\xi}' \in Y} \phi(\bm{\xi},\bm{\xi}')$ can be regarded as a penalty item.
If $P(\Xi \backslash Y)=0$, i.e. the support set of $P$ is contained in $Y$, the objective function is consistent with that in the classical DRO model \cref{DRO}. Otherwise, the value of the expectation over $\Xi \backslash Y$ is subtracted by $\int_{\Xi \backslash Y} \min_{\bm{\xi}' \in Y} \phi(\bm{\xi},\bm{\xi}') dP$. According to the construction of $Y$, core samples $\bm{\xi}_1,\ldots,\bm{\xi}_N$ are contained in $Y$. Then we can believe that the probability $P(\Xi \backslash Y)$ is extremely small. It is reasonable to reduce the influence of the expectation over $\Xi \backslash Y$. Thus, the possible worst-case expectation is decreased and the degree of conservatism is reduced.

The second model called GDRO2 embodies the globalization in a constraint,
\begin{equation}
\label{GDRO2}
\min \limits_{\bm{x} \in X} \max \limits_{P \in \mathcal{D}_1} E_{P}[h(\bm{x},\bm{\xi})],
\end{equation}
where the ambiguity set is defined as
\begin{equation}
\label{ambiguity2}
   \mathcal{D}_1=
   \left\{
   P \in \mathcal{P}
   \middle\arrowvert
   \begin{array}{lcl}
    & (E_{P}[\bm{\xi}]-\bm{\mu}_0)^{T}\bm{\Sigma}_0^{-1}(E_{P}[\bm{\xi}]-\bm{\mu}_0) \leq \gamma_1  \\
    & E_{P}[(\bm{\xi}-\bm{\mu}_0)(\bm{\xi}-\bm{\mu}_0)^{T}] \preceq \gamma_2\bm{\Sigma}_0 \\
    & E_{P}[\min \limits_{\bm{\xi}' \in Y} \phi (\bm{\xi},\bm{\xi}')] \leq d_0 \\
    & P(\bm{\xi} \in \Xi) = 1
   \end{array}
   \right\}.
\end{equation}
All notations keep the same meanings as above. It is obvious that the performance of the probability distribution on $\Xi \backslash Y$ is limited by the constraint
$$E_{P}[\min \limits_{\bm{\xi}' \in Y} \phi (\bm{\xi},\bm{\xi}')] \leq d_0.$$
It forces that the probability of that $\bm{\xi}$ is farther away from $Y$ is smaller.

The key idea of formulating the above two models is to realize the idea of globalization by introducing the distance function $\phi$. Besides, GDRO models are able to control the degree of conservatism by adjusting $\phi$. The term controlling the degree of conservatism $\min_{\bm{\xi}' \in Y} \phi (\bm{\xi},\bm{\xi}')$ appears in the objective function in GDRO1, while it is transferred to a constraint in the ambiguity set in GDRO2. Therefore, GDRO2 has one more parameter $d_0$ than GDRO1 to control conservatism.

The idea of globalization comes from the globalized robust optimization (GRO) introduced by Ben-Tal et al. \cite{globalized_ro}. Their motivation is to find a way to balance conservatism and robustness in robust optimization. Since the robust optimization was proposed, there were some articles \cite{globalized_ro, price} that aimed to discuss the trade-off between conservatism and robustness. The robustness requires constraints to be feasible for all parameter values in the uncertainty set. The resulting optimal solutions are over-conservative in the sense that decision makers give up too much optimality of the nominal problem as \cite{price} said. In extreme cases, there may be no feasible solution. In order to decrease the degree of conservatism and guarantee the existence of feasible solutions, we have to shrink the uncertainty set, resulting in the loss of robustness at the same time.

To deal with this, the GRO approach uses two uncertainty sets: an inner uncertainty set and an outer uncertainty set. Constraints in the GRO model are forced to be feasible for all parameter values in the inner uncertainty set but allowed to be infeasible for parameter values in the outer uncertainty set, where the violation is controlled by the distance of the parameter value and the inner uncertainty set. Compared with the RO approach, the GRO approach allows decision makers to consider the infection of more values of uncertain parameters even though the parameter values in the outer uncertainty set are less likely to occur in the real life. At the same time, the rise of robustness does not lead to much higher degree of conservatism.

The first published study of DRO was introduced by Scarf \cite{scarf}, but had not attracted much attention until robust optimization was popular. DRO handles the uncertainty by using some information like mean and variance of the samples under the assumption of not knowing the exact probability distribution. However, it is not so easy to estimate the support set of the probability distribution based on given samples. Generally, a set $\Phi$ known to contain the support set is used in their ambiguity sets, which is assumed to be a convex compact set, such as in \cite{ye,note_dro}. Under the assumption that $\Phi$ is bounded, the ambiguity set cannot contain the true probability distribution if its support set is the whole space $\mathbb{R}^p$. On the other side, it may lead to over-conservatism if $\Phi$ is directly taken as $\mathbb{R}^p$. In consideration of the similarity of the uncertainty set in RO and the support set in DRO, it is natural to introduce the idea of globalization to the DRO problem. That is why we propose such globalized distributionally robust optimization models.

In consideration of $\bm{\xi}(\bm{\omega})=\bm{\omega}$ for all $\bm{\omega} \in \Xi$, the support set is contained in the sample space $\Xi$. Besides, there is no problem to let the sample space to be the same as the support set if the latter is known. Hence, we choose to use the term ``sample space'' which is more rigorous in expression, instead of ``a set known to contain the support set''.

In GDRO models, we set a core set $Y$ containing core samples. In application, they can be eliminated if there are abnormal samples. In other words, the core set only include the key information which may depend on the preference of decision makers. As for the sample space, we hope it include as much information of samples as possible. Thus, we can balance the degree of robustness and conservatism by adjusting the core set $Y$, the sample space $\Xi$ and the distance function $\phi$.
Besides, each of the above GDRO models with appropriate settings of $\Xi$, $Y$ and $\phi$ can be equivalently reformulated as a semi-definite program (SDP) under the assumption that $h(\bm{x},\bm{\xi})$ is piecewise linear in $\bm{\xi}$.

There are a lot of works concentrating on DRO under moment uncertainty.
Delage and Ye \cite{ye} discuss the complexity of such DRO problem and reformulate the DRO model of a portfolio optimization problem to a SDP when the set $\Phi$ is an ellipsoid or the whole space.
Liu et al. \cite{note_dro} consider a DRO model under moment uncertainty and obtain tractable reformulations when the set $\Phi$ is the whole space or a convex polyhedral set.
Cheng et al. \cite{dro_pca} provide an approximation method based on principal component analysis for solving DRO problems with moment-based ambiguity sets, the calculation time of which is reduced in the large-scale case compared with solving SDP reformulation directly.
As for more related works, we refer to \cite{var_dro, matrix, joint_cc}.

There are also some works similar to our ideas. Li and Kwon \cite{penalty} present a penalty framework based on DRO for the portfolio problem. They assume that the investor has a particularly strong preference for a specific measure $Q$. Therefore, it is reasonable to subtract $k H(P|Q)$ from the objective function, where $H$ is the penalty function and $k$ is the penalty coefficient. In their moment-based framework, the single measure $Q$ is replaced by a set of first two moments $(\bm{\mu}_c,\bm{\Sigma}_c) = \{(\bm{\mu}_i,\bm{\Sigma}_i)| i\in C\}$ and the penalty item becomes $\mathcal{P}_k(\bm{\mu},\bm{\Sigma}|\bm{\mu}_c,\bm{\Sigma}_c)$. The SDP reformulation can be generated under appropriate choices of the confidence region $(\bm{\mu}_c,\bm{\Sigma}_c)$ and the penalty function.
Besides, other globalized distributionally robust optimization model has appeared in \cite{globalized_dro}. Different from what we care about, they apply globalization to the ambiguity set. Their main idea is to construct two ambiguity sets: an inner ambiguity set of ``normal range'' of distributions and an outer ambiguity set of ``physically possible'' distributions. Expectation constraints in their model are forced to be feasible for all distributions in the inner ambiguity set but allowed to be infeasible for distributions in the outer ambiguity set, where the violation is controlled by the distance of the distribution and the inner ambiguity set.

The remainder of this paper is organized as follows. In Section 2, we derive the computationally tractable reformulations of the two GDRO models under some assumptions. Section 3 presents some numerical results. Section 4 provides some conclusions.

\emph{Notations.}  $\bm{x}$ is the decision vector and $X$ is the feasible region of $\bm{x}$. $\Xi \subseteq \mathbb{R}^p$ is the sample space and $\mathcal{B}(\Xi)$ is the Borel $\sigma$-algebra on $\Xi$. $\mathcal{P}$ is the set of all probability measures on the measurable space $(\Xi,\mathcal{B}(\Xi))$. $\bm{\xi}$ is a random vector satisfying $\bm{\xi}(\bm{\omega})=\bm{\omega}$ in probability space $(\Xi,\mathcal{B}(\Xi),P)$. $E_P[\cdot]$ represents the expectation with respect to $\bm{\xi}$. $\bm{\mu}_0$ is the sample mean and $\bm{\Sigma}_0$ is the sample covariance matrix. $Y$ is the core set in GDRO models. $\mathbb{R}^n$ is the set of real $n$-dimensional vectors, $\mathbb{S}^n$ is the set of real $n\times n$ symmetric matrices and $\mathbb{S}^n_+$ and $\mathbb{S}^n_{++}$ are the set of positive semi-definite and positive definite matrices in $\mathbb{S}^n$, respectively. $\bm{e}_p \in \mathbb{R}^p$ is a vector of 1 and $\bm{\mathcal{I}}_p \in \mathbb{R}^{p \times p}$ is the identity matrix.

\section{Computationally tractable GDRO reformulations}
Under some assumptions, corresponding tractable SDPs are reformulated considering two settings of the sample space of GDRO models. All proofs of the lemmas and theorems in this section are left in Appendix A.
\subsection{Preparations}
It is a common way to write the dual of the original model to reformulate the problem, which is also an essential part of deriving computationally tractable GDRO reformulations. Delage and Ye \cite{ye} have proved the following lemma.

\begin{lemma}[Delage and Ye \cite{ye}]\label{dual}
Suppose $\gamma_1 >0$, $\gamma_2 >0$ and $\bm{\Sigma}_0\in \mathbb{S}^p_{++}$ in the ambiguity set \cref{ambiguity} and that $h(\bm{x},\bm{\xi})$ is integrable for all
$P \in \mathcal{D}$ for a fixed $\bm{x} \in \mathbb{R}^n$. Then DRO \cref{DRO} with the ambiguity set \cref{ambiguity} can be reformulated as
\begin{equation}\label{reDRO}
\begin{split}
\min \limits_{\bm{x},\bm{\Lambda},\bm{q},t} \quad & t+\bm{\Lambda} \cdot (\gamma_2 \bm{\Sigma}_0 +\bm{\mu}_0 \bm{\mu}_0^{T})+ \sqrt{\gamma_1}||\bm{\Sigma}_0^{1/2} (\bm{q}+2\bm{\Lambda} \bm{\mu}_0)||_2 + \bm{q}^{T} \bm{\mu}_0 \\
s.t. \quad & h(\bm{x},\bm{\xi})-\bm{\xi}^{T}\bm{\Lambda} \bm{\xi} - \bm{q}^{T} \bm{\xi} \leq t \qquad \forall \bm{\xi} \in \Xi, \\
& \bm{\Lambda} \in \mathbb{S}_+^p, \quad \bm{q} \in \mathbb{R}^p, \quad t\in\mathbb{R}, \\
& \bm{x} \in X\subseteq \mathbb{R}^n,
\end{split}
\end{equation}
where $\bm{\Sigma}_0^{1/2}$ is one decomposition of $\bm{\Sigma}_0$ such that $\bm{\Sigma}_0=(\bm{\Sigma}_0^{1/2})^T\bm{\Sigma}_0^{1/2}.$
\end{lemma}

This lemma is key for the following SDP reformulations. \cref{reDRO} decouples the min/max operations in the objective and the expectation of random vector $\bm{\xi}$. Now it is a semi-infinite programming without uncertainty.

In order to simplify the models and derive computationally tractable counterparts, we make the following assumption.
\begin{assumption}\label{plinear}
$h(\bm{x},\bm{\xi})$ is a piecewise linear concave function in $\bm{\xi}$ as following,
$$h(\bm{x},\bm{\xi})=\max \limits_{k\in \{1,2,...,K\}} \{\bm{a}_k(\bm{x})^{T} \bm{\xi}+b_k(\bm{x})\},$$
where $\bm{a}_k(\bm{x})$ and $b_k(\bm{x})$ are linear in $\bm{x}$ for all $k$.
\end{assumption}

It is easy to verify that $h(\bm{x},\bm{\xi})$ satisfies the integrable condition in \cref{dual} under this assumption. Piecewise linear function is widely used as the objective function in the portfolio optimization \cite{portfolio1,portfolio2,penalty} and inventory problems \cite{inventory1,inventory2}. Besides, some utility functions can be approximated accurately using piecewise linear functions in the portfolio optimization as stated in \cite{ye}. Especially, optimization problems with the worst case CVaR measure as the objective function can be regarded as DRO problems satisfying \cref{plinear} if the loss function $L(\bm{x},\bm{\xi})$ is linear in $\bm{\xi}$ and $\bm{x}$.
$$\text{WC-CVaR}_{\varepsilon}(L(\bm{x},\bm{\xi})) = \inf \limits_{\beta \in \mathbb{R}} \left( \beta + \frac{1}{\varepsilon} \sup \limits_{P \in \mathcal{D}} E_P[( L(\bm{x},\bm{\xi})-\beta)^+]\right).$$
Here $h(\bm{x},\bm{\xi}) = (L(\bm{x},\bm{\xi})-\beta)^+ = \max \{L(\bm{x},\bm{\xi})-\beta,0\}$.

\subsection{Globalization in objective function -- GDRO1} \label{obj_func}
In this subsection, we recall GDRO1 \cref{GDRO1} with the ambiguity set \cref{ambiguity}. For a fixed $\bm{x} \in \mathbb{R}^n$, assume that $h(\bm{x},\bm{\xi})-\min_{\bm{\xi}' \in Y} \phi(\bm{\xi},\bm{\xi}')$ is integrable for all $P \in \mathcal{D}$. According to \cref{dual}, GDRO1 \cref{GDRO1} with the ambiguity set \cref{ambiguity} and $\gamma_1 >0$, $\gamma_2 >0$, and $\bm{\Sigma}_0\in \mathbb{S}^p_{++}$ can be equivalently reformulated as
\begin{equation}\label{trac1}
\begin{split}
\min \limits_{\bm{x},\bm{\Lambda},\bm{q},t} \quad & t+\bm{\Lambda} \cdot (\gamma_2 \bm{\Sigma}_0 +\bm{\mu}_0 \bm{\mu}_0^{T})+ \sqrt{\gamma_1}||\bm{\Sigma}_0^{1/2} (\bm{q}+2\bm{\Lambda} \bm{\mu}_0)||_2 + \bm{q}^{T} \bm{\mu}_0 \\
s.t. \quad & h(\bm{x},\bm{\xi})-\bm{\xi}^{T}\bm{\Lambda} \bm{\xi} - \bm{q}^{T} \bm{\xi}-t \leq \min \limits_{\bm{\xi}' \in Y} \phi(\bm{\xi},\bm{\xi}') \qquad \forall \bm{\xi} \in \Xi, \\
& \bm{\Lambda} \in \mathbb{S}^p_+, \quad \bm{q} \in \mathbb{R}^p, \quad t\in\mathbb{R} ,\\
& \bm{x} \in X\subseteq\mathbb{R}^n.
\end{split}
\end{equation}

Especially, if we have no idea about the estimation of the sample space, then it is reasonable to set
\begin{equation}\label{unbounded}
\begin{split}
& Y = \{\bm{\xi} \in \mathbb{R}^p |
    (\bm{\xi}-\bar{\bm{\mu}})^{T}\bm{\Sigma}_0^{-1}(\bm{\xi}-\bar{\bm{\mu}}) \leq \bar{\gamma} \}, \\
& \Xi = \mathbb{R}^p.
\end{split}
\end{equation}

\begin{theorem}\label{thm1}
Let $\phi(\bm{\xi},\bm{\xi}') = \theta \|\bm{\xi}-\bm{\xi}'\|_2$ for a given $\theta\ge 0$ and $Y$ and $\Xi$ be defined as \cref{unbounded}. Under \cref{plinear},  \cref{trac1}
is equivalent to the following SDP problem when $\gamma_1 >0$, $\gamma_2 >0$ and $\bm{\Sigma}_0\in \mathbb{S}^p_{++}$,
\begin{equation}\label{tracGDRC_1_Rp}
\begin{split}
\min \limits_{\bm{x},\bm{\Lambda},\bm{q},t,\bm{v}_k,z_k} & t+\bm{\Lambda} \cdot (\gamma_2 \bm{\Sigma}_0 +\bm{\mu}_0 \bm{\mu}_0^{T})+ \sqrt{\gamma_1}||\bm{\Sigma}_0^{1/2} (\bm{q}+2\bm{\Lambda}\bm{\mu}_0)||_2 + \bm{q}^{T} \bm{\mu}_0 \\
s.t.\quad & \begin{bmatrix}
      \bm{\Lambda} & \frac{1}{2}(\bm{v}_k +\bm{q}-\bm{a}_k(\bm{x})) \\
      \frac{1}{2}(\bm{v}_k +\bm{q}-\bm{a}_k(\bm{x}))^{T} & t-b_k(\bm{x})-\bar{\bm{\mu}}^{T} \bm{v}_k -z_k\\
      \end{bmatrix}  \in \mathbb{S}^{p+1}_+, \\
      &\sqrt{\bar{\gamma}} \|\bm{\Sigma}_0^{1/2}\bm{v}_k  \|_2\le z_k, \\
       & \|\bm{v}_k\|_2 \leq \theta, \\
       & \bm{v}_k \in \mathbb{R}^p, \quad z_k \in \mathbb{R}, \quad 1\le k\le K, \\
       & \bm{\Lambda} \in \mathbb{S}^p_+, \quad \bm{q}\in \mathbb{R}^p, \quad t \in \mathbb{R}, \\
       & \bm{x} \in X\subseteq\mathbb{R}^n.
\end{split}
\end{equation}
\end{theorem}

However, the case that the sample space is known to be bounded may also arise in the real life. We may not be able to know the exact shape and size of the sample space. In this case, the idea of globalization also make sense by setting
\begin{equation}\label{bounded}
\begin{split}
& Y = \{\bm{\xi} \in \mathbb{R}^p |
    (\bm{\xi}-\bar{\bm{\mu}})^{T}\bm{\Sigma}_0^{-1}(\bm{\xi}-\bar{\bm{\mu}}) \leq \bar{\gamma}_1 \},\\
& \Xi = \{\bm{\xi} \in \mathbb{R}^p |
    (\bm{\xi}-\bar{\bm{\mu}})^{T}\bm{\Sigma}_0^{-1}(\bm{\xi}-\bar{\bm{\mu}}) \leq \bar{\gamma}_2 \},\\
\end{split}
\end{equation}
where $0 < \bar{\gamma}_1 < \bar{\gamma}_2$.

\begin{theorem}\label{thm2}
Let $\phi(\bm{\xi},\bm{\xi}') = \theta \|\bm{\xi}-\bm{\xi}'\|_2$ for a given $\theta\ge 0$ and $Y$ and $\Xi$ be defined as \cref{bounded}. Under \cref{plinear}, \cref{trac1}
is equivalent to the following SDP problem when $\gamma_1 >0$, $\gamma_2 >0$, and $\bm{\Sigma}_0\in \mathbb{S}^p_{++}$,
{\small
\begin{equation}\label{tracGDRC_1_bounded}
\begin{split}
\min \limits_{\begin{array}{c}\bm{x},\bm{\Lambda},\bm{q},t\\ \bm{v}_k,\bm{w}_k\\ z_{1k},z_{2k} \end{array}} & t+\bm{\Lambda} \cdot (\gamma_2 \bm{\Sigma}_0 +\bm{\mu}_0 \bm{\mu}_0^{T})+ \sqrt{\gamma_1}||\bm{\Sigma}_0^{1/2} (\bm{q}+2\bm{\Lambda}\bm{\mu}_0)||_2 + \bm{q}^{T} \bm{\mu}_0 \\
s.t.\quad & \left[\begin{array}{cc}
      \bm{\Lambda} & \frac{1}{2} (\bm{v}_k + \bm{w}_k +\bm{q}-\bm{a}_k(\bm{x}))\\
      \frac{1}{2}(\bm{v}_k + \bm{w}_k +\bm{q}-\bm{a}_k(\bm{x}))^{T} & t-b_k(\bm{x})-\bar{\bm{\mu}}^{T} (\bm{v}_k + \bm{w}_k) - z_{1k} - z_{2k} \\
      \end{array} \right]\in S^{p+1}_+ , \\
       & \sqrt{\bar{\gamma}_1} \|\bm{\Sigma}_0^{1/2}\bm{v}_k \|_2 \leq z_{1k}, \\
       & \sqrt{\bar{\gamma}_2} \|\bm{\Sigma}_0^{1/2}\bm{w}_k \|_2 \leq z_{2k}, \\
       & \|\bm{v}_k\|_2 \leq \theta, \\
       & \bm{v}_{k}, \bm{w}_{k}\in \mathbb{R}^p, \quad z_{1k},z_{2k}\in\mathbb{R}, \quad 1\le k\le K,\\
       & \bm{\Lambda} \in \mathbb{S}^p_+, \quad \bm{q}\in \mathbb{R}^p, \quad t\in\mathbb{R},\\
       & \bm{x} \in X\subseteq\mathbb{R}^n.
\end{split}
\end{equation}
}
\end{theorem}

It is obvious that we are able to control the degree of conservatism by adjusting the coefficient $\theta$ in the distance function $\phi(\bm{\xi},\bm{\xi}') = \theta \|\bm{\xi}-\bm{\xi}'\|_2$. If $\theta=0$, the two GDRO models degenerate into the classical DRO models. In \cref{tracGDRC_1_bounded}, we observe that it forces $\bm{w}_k=0$ and $z_{2k}=0$ for all $k$ when the parameter $\bar{\gamma}_2$ tends to be infinite, resulting in \cref{tracGDRC_1_bounded} to be the same as \cref{tracGDRC_1_Rp}.

\subsection{Globalization in ambiguity set -- GDRO2} \label{ambi_set}
In this subsection, we recall GDRO2 \cref{GDRO2} with the ambiguity set \cref{ambiguity2}. Similarly to \cref{dual}, under assumptions that $\gamma_1 >0$, $\gamma_2 >0$ and $\bm{\Sigma}_0\in \mathbb{S}^p_{++}$ and $h(\bm{x},\bm{\xi})$ is integrable for all $P \in \mathcal{D}$, the GDRO \cref{GDRO2} with the ambiguity set \cref{ambiguity2} for a fixed $\bm{x} \in \mathbb{R}^n$  can be reformulated as
\begin{equation}\label{trac2}
\begin{split}
\min \limits_{\bm{x},\bm{\Lambda},\bm{q},t,r} \quad & t+\bm{\Lambda} \cdot (\gamma_2 \bm{\Sigma}_0 +\bm{\mu}_0 \bm{\mu}_0^{T})+ \sqrt{\gamma_1}||\bm{\Sigma}_0^{1/2} (\bm{q}+2\bm{\Lambda} \bm{\mu}_0)||_2 + \bm{q}^{T} \bm{\mu}_0 + d_0 r\\
s.t. \quad & h(\bm{x},\bm{\xi})-\bm{\xi}^{T}\bm{\Lambda} \bm{\xi} - \bm{q}^{T} \bm{\xi}-t \leq r \min \limits_{\bm{\xi}' \in Y} \phi(\bm{\xi},\bm{\xi}') \qquad \forall \bm{\xi} \in \Xi, \\
& \bm{\Lambda} \in \mathbb{S}_+^p, \quad \bm{q} \in \mathbb{R}^p, \quad t\in\mathbb{R}, \quad r \geq 0, \\
& \bm{x} \in X\subseteq \mathbb{R}^n.
\end{split}
\end{equation}

Similarly, we have conclusions as follows for the case \cref{unbounded,bounded}.

\begin{theorem}\label{thm3}
Let $\phi(\bm{\xi},\bm{\xi}') = \theta \|\bm{\xi}-\bm{\xi}'\|_2$ for a given $\theta\ge 0$ and $Y$ and $\Xi$ be defined as \cref{unbounded}. Under \cref{plinear}, \cref{trac2}
is equivalent to the following SDP problem when $\gamma_1 >0$, $\gamma_2 >0$ and $\bm{\Sigma}_0 \in \mathbb{S}^p_{++}$,
%{\small
\begin{equation}\label{tracGDRC_2_Rp}
\begin{split}
\min \limits_{\bm{x},\bm{\Lambda},\bm{q},t,\bm{v}_k,z_k,r} \quad & t+\bm{\Lambda} \cdot (\gamma_2 \bm{\Sigma}_0 +\bm{\mu}_0 \bm{\mu}_0^{T})+ \sqrt{\gamma_1}||\bm{\Sigma}_0^{1/2} (\bm{q}+2\bm{\Lambda}\bm{\mu}_0)||_2 + \bm{q}^{T} \bm{\mu}_0 + d_0 r \\
s.t. & \begin{bmatrix}
      \bm{\Lambda} & \dfrac{1}{2}(\bm{v}_k +\bm{q}-\bm{a}_k(\bm{x})) \\
      \dfrac{1}{2}(\bm{v}_k +\bm{q}-\bm{a}_k(\bm{x}))^{T} & t-b_k(\bm{x})-\bar{\bm{\mu}}^{T} \bm{v}_k - z_k\\
      \end{bmatrix}\in \mathbb{S}^{p+1}_+ ,\\
       & \sqrt{\bar{\gamma}} \|\bm{\Sigma}_0^{1/2}\bm{v}_k \|_2 \leq z_k, \\
       & \|\bm{v}_k\|_2 \leq r \theta, \\
       & \bm{v}_k \in \mathbb{R}^p, \quad z_k\in\mathbb{R}, \quad 1\le k\le K, \\
       & \bm{\Lambda} \in \mathbb{S}^p_+, \quad \bm{q}\in \mathbb{R}^p, \quad t\in\mathbb{R}, \quad r \geq 0,\\
       & \bm{x} \in X\subseteq \mathbb{R}^n.
\end{split}
\end{equation}
%}
\end{theorem}

\begin{theorem}\label{thm4}
Let $\phi(\bm{\xi},\bm{\xi}') = \theta \|\bm{\xi}-\bm{\xi}'\|_2$ for a given $\theta\ge 0$ and $Y$ and $\Xi$ be defined as \cref{bounded}. Under \cref{plinear}, \cref{trac2}
is equivalent to the following SDP problem when $\gamma_1 >0$, $\gamma_2 >0$ and $\bm{\Sigma}_0\in \mathbb{S}^p_{++}$,
%{\small
\begin{equation}\label{tracGDRC_2_bounded}
\begin{split}
\min \limits_{\mbox{\tiny$\begin{array}{c}\bm{x},\bm{\Lambda},\bm{q},t,r,\\ \bm{v}_k,\bm{w}_k\\
z_{1k},z_{2k} \end{array}$}} &  t+\bm{\Lambda} \cdot (\gamma_2 \bm{\Sigma}_0 +\bm{\mu}_0 \bm{\mu}_0^{T})+ \sqrt{\gamma_1}||\bm{\Sigma}_0^{1/2} (\bm{q}+2\bm{\Lambda}\bm{\mu}_0)||_2 + \bm{q}^{T} \bm{\mu}_0 + d_0r \\
s.t.\quad &{ \left[\begin{array}{cc}
      \bm{\Lambda} & \frac{1}{2} (\bm{v}_k + \bm{w}_k +\bm{q}-\bm{a}_k(\bm{x}))\\
      \frac{1}{2}(\bm{v}_k + \bm{w}_k +\bm{q}-\bm{a}_k(\bm{x}))^{T} & t-b_k(\bm{x})-\bar{\bm{\mu}}^{T} (\bm{v}_k + \bm{w}_k) - z_{1k} - z_{2k} \\
      \end{array} \right]\in \mathbb{S}^{p+1}_+},\\
       & \sqrt{\bar{\gamma_1}} \|\bm{\Sigma}_0^{1/2}\bm{v}_k \|_2 \leq z_{1k}, \\
       & \sqrt{\bar{\gamma_2}} \|\bm{\Sigma}_0^{1/2}\bm{w}_k \|_2 \leq z_{2k}, \\
       & \|\bm{v}_k\|_2 \leq r\theta, \\
       & \bm{v}_{k}, \bm{w}_{k}\in \mathbb{R}^p, \quad z_{1k},z_{2k}\in\mathbb{R}, \quad 1\le k\le K,\\
       & \bm{\Lambda} \in \mathbb{S}^p_+, \quad \bm{q}\in \mathbb{R}^p, \quad t\in\mathbb{R}, \quad r \geq 0,\\
       & \bm{x} \in X\subseteq \mathbb{R}^n.
\end{split}
\end{equation}
%}
\end{theorem}

For the two GDRO models with two settings \cref{unbounded,bounded} of $Y$ and $\Xi$, we get four SDP problems \cref{tracGDRC_1_Rp,tracGDRC_1_bounded,tracGDRC_2_Rp,tracGDRC_2_bounded}, which are polynomially solvable.
Then numerical experiments can be adopted to evaluate the optimal objective values of the problems influenced by the control parameter $\theta$ and the two settings.
The GDRO models with setting \cref{unbounded} are referred as the unbounded case, while the GDRO models with setting \cref{bounded} are referred as the bounded case in the following section.

\section{Numerical results}
The portfolio optimization problem is used in our numerical experiments which has been previously studied by Delage and Ye \cite{ye} in this section. Given $n$ options, $\bm{\xi}$ is a random vector of returns and $E_P[u(\bm{\xi}^T \bm{x})]$ is the expected utility with respect to the probability distribution $P$. Given the ambiguity set $\mathcal{D}$, the distributionally robust portfolio optimization model is written as
\begin{equation}\label{portfolio}
\max \limits_{\bm{x} \in X} \min \limits_{P \in \mathcal{D}} E_P[u(\bm{\xi}^T \bm{x})],
\end{equation}
where
$$X = \{\bm{x} \in \mathbb{R}^n | \sum \limits_{i=1}^n x_i = 1, \bm{x} \geq 0\}.$$

Here the same utility function as the one in \cite{ye} is considered, which is a piecewise linear function.
$$u(\bm{\xi}^T \bm{x})= \min \limits_{k \in \{1,2,\ldots,K\}} \{a_k \bm{\xi}^{T} \bm{x} + b_k\}.$$

Reformulate problem \cref{portfolio} as a standard form:
\begin{equation}
\min \limits_{\bm{x} \in X} \max \limits_{P \in \mathcal{D}} E_P[\max \limits_{k \in \{1,2,\ldots,K\}} - a_k \bm{\xi}^{T} \bm{x} - b_k].
\end{equation}

In our globalized distributionally robust optimization models, we set
$$\phi(\bm{\xi},\bm{\xi}') = \theta \| \bm{\xi}-\bm{\xi}' \|_2.$$

In the numerical experiments, all data are randomly generated without considering the real background of the portfolio optimization problem. We take $n = p = 10$, $K = 5$ and $a_k,\ b_k$ being data of independently uniform distribution on $(0,10)$. In the following subsections of numerical experiments, 30 samples are generated for the uncertain $\bm{\xi}$ of each configuration according to a given distribution and a small perturbation $N(\epsilon \bm{e}_p, 0.5\bm{\mathcal{I}}_p)$, where $\epsilon \in (-0.5,0.5)$, $\bm{e}_p \in \mathbb{R}^p$ is a vector of 1 and $\bm{\mathcal{I}}_p \in \mathbb{R}^{p \times p}$ is the identity matrix. The small perturbation added to the samples aims to make difference between the means and variances from the computation and given distribution.

The sample mean $\bm{\mu}_0$ and the sample covariance matrix $\bm{\Sigma}_0$ are computed based on the samples. Theoretical determination of $\gamma_1$ and $\gamma_2$ is given in Appendix B. Here $\gamma_1$ and $\gamma_2$ as defined by \cref{normal_gamma} with $\alpha = 0.05$ are used and computed in all numerical experiments due to the small sample size. The other parameters may have different settings in different experiments. For convenience, notations in \cref{model} are used to represent the models in the the following numerical experiments.

\begin{table}[tbhp]
\renewcommand\arraystretch{1.5}
\centering
\caption{Models in numerical experiments}\label{model}
\setlength{\tabcolsep}{2mm}{
\resizebox{\textwidth}{25mm}{
\begin{tabular}{c|ccccc}
\hline
notation & model &	sample space $\Xi$ & core set $Y$ & optimal objective value \\
\hline
DRO1 & $\min \limits_{\bm{x} \in X} \max \limits_{P \in \mathcal{D}} E_{P}[h(\bm{x},\bm{\xi})]$ & $\mathbb{R}^p$ & - & $v(DRO1)$ \\
DRO2 & $\min \limits_{\bm{x} \in X} \max \limits_{P \in \mathcal{D}} E_{P}[h(\bm{x},\bm{\xi})]$ & $\mathcal{E}(\gamma)$ & - & $v(DRO2,\gamma)$ \\
GDRO1.1 & $\min \limits_{\bm{x} \in X} \max \limits_{P \in \mathcal{D}} E_{P}[h(\bm{x},\bm{\xi})-\min \limits_{\bm{\xi}' \in Y} \phi(\bm{\xi},\bm{\xi}')]$ & $\mathbb{R}^p$ & $\mathcal{E}(\bar{\gamma})$ & $v(GDRO1.1,\bar{\gamma})$ \\
GDRO1.2 & $\min \limits_{\bm{x} \in X} \max \limits_{P \in \mathcal{D}} E_{P}[h(\bm{x},\bm{\xi})-\min \limits_{\bm{\xi}' \in Y} \phi(\bm{\xi},\bm{\xi}')]$ & $\mathcal{E}(\bar{\gamma}_2)$ & $\mathcal{E}(\bar{\gamma}_1)$ & $v(GDRO1.2,\bar{\gamma}_1,\bar{\gamma}_2)$ \\
GDRO2.1 & $\min \limits_{\bm{x} \in X} \max \limits_{P \in \mathcal{D}_1} E_{P}[h(\bm{x},\bm{\xi})]$
& $\mathbb{R}^p$ & $\mathcal{E}(\bar{\gamma})$ & $v(GDRO2.1,\bar{\gamma})$\\
GDRO2.2 & $\min \limits_{\bm{x} \in X} \max \limits_{P \in \mathcal{D}_1} E_{P}[h(\bm{x},\bm{\xi})]$
& $\mathcal{E}(\bar{\gamma}_2)$ & $\mathcal{E}(\bar{\gamma}_1)$ & $v(GDRO2.2,\bar{\gamma}_1,\bar{\gamma}_2)$ \\
\hline
\end{tabular}
}}
\begin{tablenotes}
\item \small \emph{Note}. $\mathcal{E}(\gamma) = \{\bm{\xi} \in \mathbb{R}^p |(\bm{\xi}-\bar{\bm{\mu}})^{T}\bm{\Sigma}_0^{-1}(\bm{\xi}-\bar{\bm{\mu}}) \leq \gamma \}$.
\end{tablenotes}
\end{table}

The influence of the parameter $\theta$ and the size of the core set or sample space on the optimal objective value is considered in \cref{theta-value,radius_unbounded,bounded_value}. Discussions on the determination of the core set, the sample space and the control parameters are presented in the \cref{data_based} based on the sample data.

\subsection{Selection of \texorpdfstring{$\theta$}{theta} and sample space}\label{theta-value}
The uniform distribution on the box $[0,20]^p$ is chosen as the true probability distribution. Then 30 samples with the small perturbation are generated. Let $\bar{\bm{\mu}} = \bm{\mu}_0$ and $\bar{\gamma} = \gamma_1$ in \cref{unbounded}, $\bar{\bm{\mu}} = \bm{\mu}_0$, $\bar{\gamma}_1 = \gamma_1$ and $\bar{\gamma}_2 = 100 \gamma_1$ in \cref{bounded} and $d_0=0.5$ in \cref{ambiguity2}. The variation of the optimal objective value of GDRO models with the parameter $\theta$ is shown in \cref{fig:theta}.

\begin{figure}[tbhp]
\centering
\subfloat[GDRO1]{\label{fig:gdro1}\includegraphics[scale=0.5]{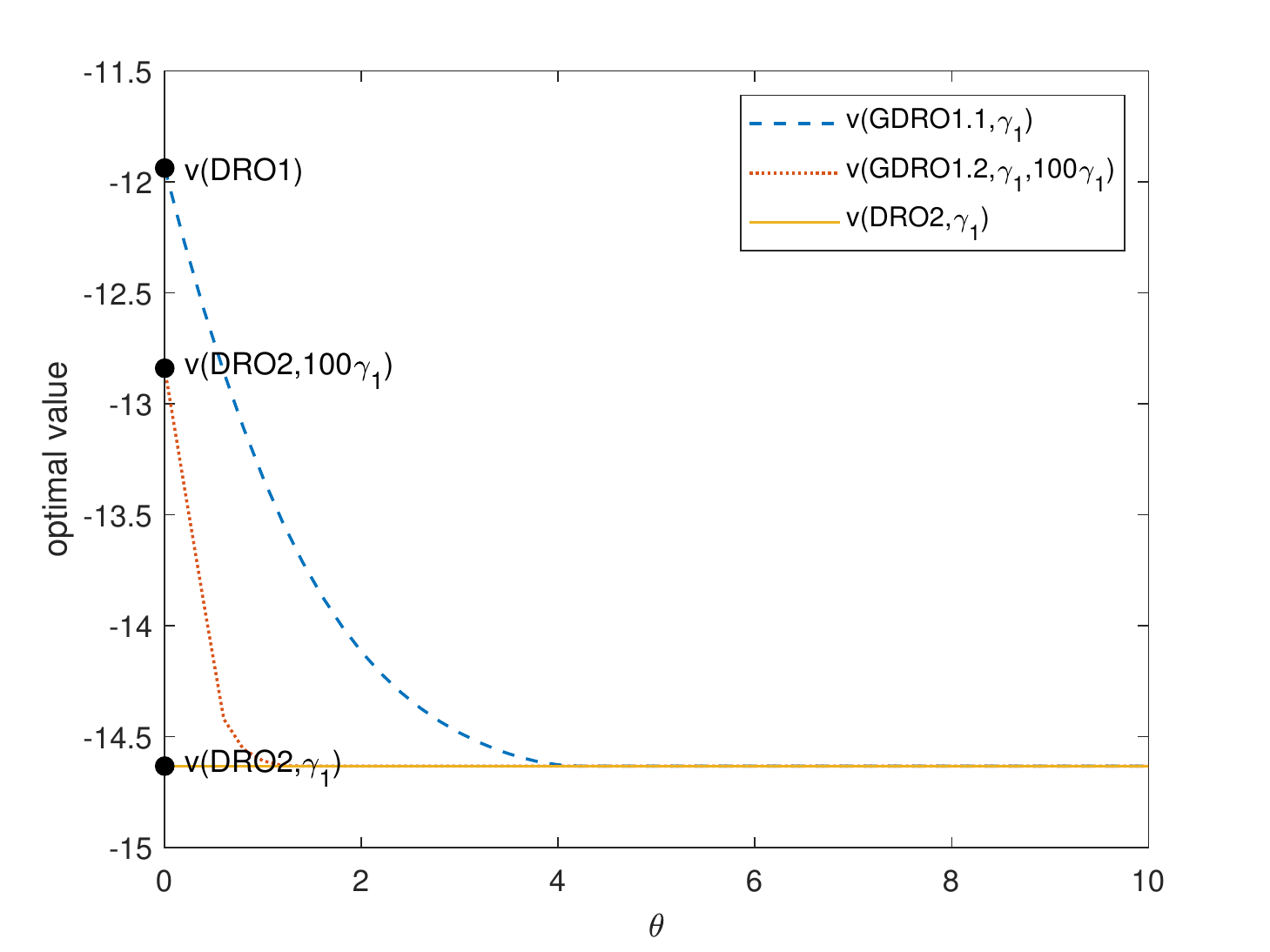}}
\subfloat[GDRO2]{\label{fig:gdro2}\includegraphics[scale=0.5]{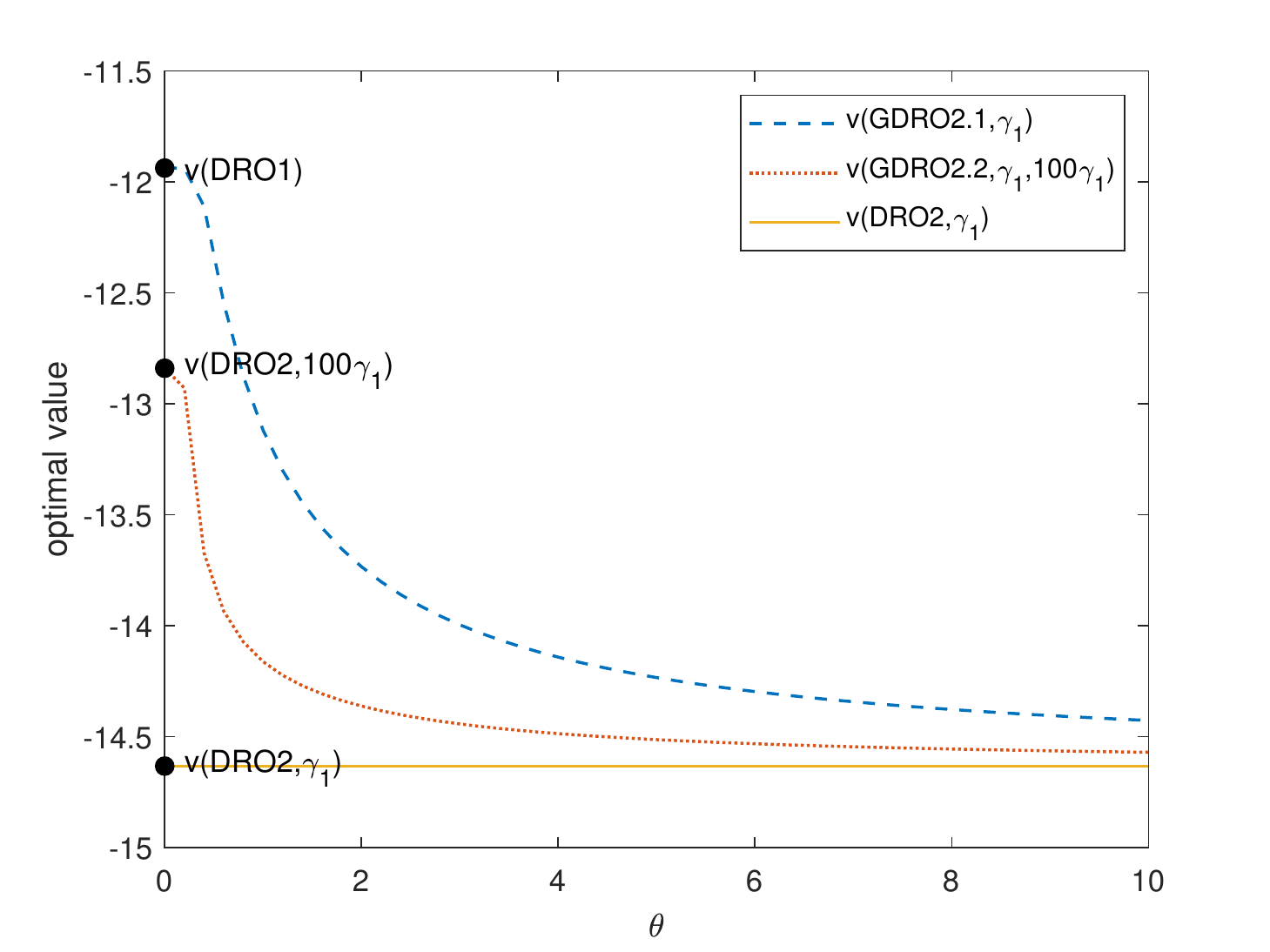}}
\caption{The optimal objective values of GDRO models via $\theta$}
\label{fig:theta}
\end{figure}

%\begin{figure}[th]
%\centering
%\includegraphics[scale=0.4]{theta_v2.jpg}
%\caption{The optimal objective values of GDRO2 via $\theta$}
%\label{fig_theta2}
%\end{figure}

In \cref{fig:theta}, $v(DRO1)$, $v(DRO2,100\gamma_1)$ and $v(DRO2,\gamma_1)$ are optimal objective values of the DRO model with sample spaces $\mathbb{R}^p$, $\mathcal{E}(100\gamma_1)$ and  $\mathcal{E}(\gamma_1)$, respectively. Thus, it can be seen that the optimal objective value decreases obviously as the size of sample space decreases. The smaller the sample space is, the better the optimal objective value and the lower degree of conservatism are.

It is obvious that the optimal objective value of GDRO models is decreased with the increase of $\theta$. In \cref{fig:gdro1}, $v(GDRO1.1,\gamma_1)=v(DRO1)$ and $v(GDRO1.2,\gamma_1,100\gamma_1)=v(DRO2,100\gamma_1)$ when $\theta = 0$. Actually, GDRO1.1 degenerates into DRO1 and GDRO1.2 degenerates into DRO2 with $\gamma = 100 \gamma_1$ at this time. In other words, the core set does not work in this situation. When $\theta$ becomes large enough, $v(GDRO1.1,\gamma_1)=v(GDRO1.2,\gamma_1,100\gamma_1)=v(DRO2,\gamma_1)$. The worst case solution\footnote{The worst case solution of the inner maximization problem of \cref{GDRO1} is $\mathop{\arg\max} \limits_{P \in \mathcal{D}} E_{P}[h(\bm{x},\bm{\xi})-\min \limits_{\bm{\xi}' \in Y} \phi(\bm{\xi},\bm{\xi}')]$.} of the inner maximization problem of GDRO1 \cref{GDRO1} may be a probability measure $P$ whose support set is contained in the core set $Y=\mathcal{E}(\gamma_1)$. Hence, GDRO1.1 and GDRO1.2 will have the same performance as the model DRO2 with $\gamma = \gamma_1$ if $\theta$ is large enough.

\Cref{fig:gdro2} has the same trend as \cref{fig:gdro1}. In \cref{fig:gdro2}, it is also the same as that in \cref{fig:gdro1} when $\theta=0$. However, for the model GDRO2.1 and GDRO2.2, the optimal objective value will be closed to $v(DRO2,\gamma_1)$ but will never reach it with $\theta$ increasing.

The decision maker can control the degree of conservatism by choosing the sample space and the value of $\theta$.
For the optimal objective values of GDRO models, the supremum (i.e. the least conservative objective value) is $v(DRO2,\gamma_1)$ while the infimum (i.e. the most conservative objective value) is $v(DRO1)$ ($v(DRO2,100\gamma_1)$) for unbounded case (bounded case) in the range where the degree of conservatism can be adjusted. Besides, the gap between the supremum and the infimum determines the size of the adjustment range of $\theta$. It also can be seen that $v(GDRO1.1,\gamma_1)$ and $v(GDRO1.2,\gamma_1,100\gamma_1)$ reach a same value with a relatively smaller $\theta$ for the later. Therefore, the adjustment range of $\theta$ for the bounded case is smaller than that for the unbounded case.

Actually, the determination of $\theta$ in GDRO models depends on the optimal value range of the objective function and the preference of decision makers. For this instance, the optimal objective value of GDRO1.1 reaches $v(DRO2,\gamma_1)$ at $\theta=4.4$ and that of GDRO1.2 reaches $v(DRO2,\gamma_1)$ at $\theta=1.4$. For general problems, we can solve the following problem to get an appropriate value of $\theta$ for GDRO1. The same also works for GDRO2 by replacing $h(\bm{x},\bm{\xi})-\min_{\bm{\xi}' \in Y} \phi(\bm{\xi},\bm{\xi}')$ with $h(\bm{x},\bm{\xi})$ and the ambiguity set $\mathcal{D}$ \cref{ambiguity} with $\mathcal{D}_1$ \cref{ambiguity2}.
\begin{equation}\label{theta}
\begin{split}
\min & \quad \theta \\
s.t. & \quad \min \limits_{\bm{x} \in \mathcal{X}} \max \limits_{P \in \mathcal{D}} E_{P}[h(\bm{x},\bm{\xi})-\min \limits_{\bm{\xi}' \in Y} \phi(\bm{\xi},\bm{\xi}')]  \leq v + \varepsilon,
\end{split}
\end{equation}
where $v=v(DRO2,\bar{\gamma})$ and $\varepsilon$ depends on the preference of decision makers. Under our assumptions, \cref{theta} is also equivalently reformulated as a SDP model which is computationally tractable.

\subsection{Selection of core set \texorpdfstring{$Y$}{Y} under sample space \texorpdfstring{$\mathbb{R}^p$}{Rp}}\label{radius_unbounded}
In this subsection, the optimal objective values of GDRO1.1 and GDRO2.1 are used to show the influence of the radius of the core set when the sample space is $\mathbb{R}^p$. The optimal objective value of DRO2 is selected as a lower bound for that of GDRO1.1 and GDRO2.1.
Each set of 30 samples are got from each of the following three distributions with the small perturbation $N(\epsilon \bm{e}_p,0.5\bm{\mathcal{I}}_p)$, the normal $N(10 \bm{e}_p,\frac{100}{3} \bm{\mathcal{I}}_p)$, uniform on box $[0,20]^p$ and uniform on ellipsoid $\{\bm{\xi}|(\bm{\xi}-10 \bm{e}_p)^{T} \bm{\Sigma}^{-1} (\bm{\xi}-10 \bm{e}_p) \leq 200 \}$, respectively, where $\bm{\Sigma}$ is a positive definite matrix.

Let $\bm{\bar{\mu}}=\bm{\mu}_0$ and $\gamma$ and $\bar{\gamma}$ be 0.25, 0.5, 0.75, 1, 3, 25, 49, 100, 225, 400, 625 times of $\gamma_1$, respectively. We repeat the experiments 100 times and compute the average and the variance of optimal objective values. \Cref{normal,box,ellipsoid} show the experiment results of DRO2, GDRO1.1 with $\theta=4$, and GDRO2.1 with $d_0 = 0.5,\theta=4$. Besides, the average of $\theta$ solved by \cref{theta} with $\varepsilon = 0.01$ is presented in \cref{theta-radius}. The optimal objective value of DRO1 is not influenced by the radius because its sample space is the whole space. Hence, the average and variance of $v(DRO1)$ are recorded as a note under every table.

\begin{table}[th]
\centering
\caption{Normal distribution $N(10 \bm{e}_p,\frac{100}{3} \bm{\mathcal{I}}_p)$ }\label{normal}
%\begin{threeparttable}
\setlength{\tabcolsep}{1mm}{
\resizebox{\textwidth}{16mm}{
\begin{tabular}{c|c|cccccccccccc}
\hline
\multicolumn{2}{c|}{$\gamma/\gamma_1$,$\bar{\gamma}/\gamma_1$} & 0.25 &	0.5 & 0.75 & 1 & 9 & 25 & 49 & 100 & 225 &	400 & 625 \\
\hline
\multirow{3}*{Mean} & DRO2    & -14.3856 & -14.0043 & -13.7411 & -13.5306 & -13.5306 & -13.4130 & -11.6106 & -9.7738 & -9.4582 & -9.4538 & -9.4538 \\
~ & GDRO1.1 & -14.3607 & -13.9972 & -13.7383 & -13.5304 & -13.5304 & -13.4130 & -11.6106 & -9.7738 & -9.4582 & -9.4538 & -9.4538 \\
~ & GDRO2.1 & -13.9004 & -13.5552 & -13.3076 & -13.1463 & -13.1422 & -13.0290 & -11.3229 & -9.6769 & -9.4548 & -9.4538 & -9.4538 \\
\hline
\multirow{3}*{Var} & DRO2   & 0.2360 & 0.2333 & 0.2385 & 0.2465 & 0.2465 & 0.7192 & 4.9440 & 3.7592 & 2.7055 & 2.6813 & 2.6813 \\
~ & GDRO1.1  & 0.2184 & 0.2265 & 0.2353 & 0.2462 & 0.2461 & 0.7192 & 4.9440 & 3.7592 & 2.7055 & 2.6813 & 2.6813 \\
~ & GDRO2.1  & 0.2126 & 0.2186 & 0.2280 & 0.2394 & 0.2392 & 0.6867 & 4.5862 & 3.4638 & 2.6867 & 2.6813 & 2.6813 \\
\hline
\end{tabular}  }}
\begin{tablenotes}
\item \small \emph{Note}. $\text{mean}(v(DRO1)) = -9.4538$, $\text{var}(v(DRO1)) = 2.7084$.
\end{tablenotes}
%\end{threeparttable}
\end{table}

\begin{table}[th]
\centering
\caption{Uniform distribution on box $[0,20]^p$}\label{box}
%\begin{threeparttable}
\setlength{\tabcolsep}{1mm}{
\resizebox{\textwidth}{16mm}{
\begin{tabular}{c|c|cccccccccccc}
\hline
\multicolumn{2}{c|}{$\gamma/\gamma_1$,$\bar{\gamma}/\gamma_1$} & 0.25 &	0.5 & 0.75 & 1 & 9 & 25 & 49 & 100 & 225 &	400 & 625 \\
\hline
\multirow{3}*{Mean} & DRO2    & -14.3539 & -13.9644 & -13.6961 & -13.4818 & -13.4818 & -13.2725 & -11.3096 & -9.4861 & -9.2181 & -9.2159 & -9.2159 \\
~ & GDRO1.1 & -14.3283 & -13.9559 & -13.6922 & -13.4813 & -13.4812 & -13.2721 & -11.3096 & -9.4861 & -9.2181 & -9.2159 & -9.2159 \\
~ & GDRO2.1 & -13.8785 & -13.5271 & -13.2750 & -13.1087 & -13.1050 & -12.9009 & -11.0399 & -9.4026 & -9.2161 & -9.2159 & -9.2159 \\
\hline
\multirow{3}*{Var} & DRO2      & 0.2193 & 0.2095 & 0.2140 & 0.2220 & 0.2220 & 1.4371 & 5.3689 & 3.8619 & 2.9436 & 2.9309 & 2.9309 \\
~ & GDRO1.1  & 0.1807 & 0.1940 & 0.2066 & 0.2209 & 0.2208 & 1.4358 & 5.3689 & 3.8619 & 2.9436 & 2.9309 & 2.9309 \\
~ & GDRO2.1  & 0.1817 & 0.1861 & 0.1961 & 0.2098 & 0.2094 & 1.3661 & 5.0032 & 3.6102 & 2.9322 & 2.9309 & 2.9309 \\
\hline
\end{tabular}  }}
\begin{tablenotes}
\item \small \emph{Note}. $\text{mean}(v(DRO1)) = -9.2159$, $\text{var}(v(DRO1)) = 2.9605$.
\end{tablenotes}
%\end{threeparttable}
\end{table}

\begin{table}[th]
\centering
\caption{Uniform distribution on ellipsoid $\{\bm{\xi}|(\bm{\xi}-10 \bm{e}_p)^{T} \bm{\Sigma}^{-1} (\bm{\xi}-10 \bm{e}_p) \leq 200 \}$}\label{ellipsoid}
%\begin{threeparttable}
\setlength{\tabcolsep}{1mm}{
\resizebox{\textwidth}{16mm}{
\begin{tabular}{c|c|cccccccccccc}
\hline
\multicolumn{2}{c|}{$\gamma/\gamma_1$,$\bar{\gamma}/\gamma_1$} & 0.25 &	0.5 & 0.75 & 1 & 9 & 25 & 49 & 100 & 225 &	400 & 625 \\
\hline
\multirow{3}*{Mean} & DRO2 & -14.4224 & -14.2647 & -14.1542 & -14.0650 & -14.0650 & -14.0650 & -14.0650 & -14.0599 & -13.7621 & -13.4019 & -13.3045 \\
~ & GDRO1.1 & -14.4158 & -14.2625 & -14.1532 & -14.0649 & -14.0649 & -14.0648 & -14.0648 & -14.0596 & -13.7620 & -13.4019 & -13.3045 \\
~ & GDRO2.1 & -14.0284 & -13.8918 & -13.7908 & -13.7441 & -13.7369 & -13.7290 & -13.7201 & -13.7001 & -13.4563 & -13.2882 & -13.2844 \\
\hline
\multirow{3}*{Var} & DRO2     & 0.0669 & 0.0664 & 0.0677 & 0.0695 & 0.0695 & 0.0695 & 0.0695 & 0.0744 & 0.2595 & 0.2391 & 0.1957 \\
~ & GDRO1.1  & 0.0652 & 0.0660 & 0.0677 & 0.0696 & 0.0696 & 0.0696 & 0.0696 & 0.0745 & 0.2595 & 0.2391 & 0.1957 \\
~ & GDRO2.1  & 0.0647 & 0.0675 & 0.0702 & 0.0729 & 0.0731 & 0.0732 & 0.0735 & 0.0780 & 0.1975 & 0.1791 & 0.1764 \\
\hline
\end{tabular}  }}
\begin{tablenotes}
\item \small \emph{Note}. $\text{mean}(v(DRO1)) = -13.2844$, $\text{var}(v(DRO1)) = 0.1782$.
\end{tablenotes}
%\end{threeparttable}
\end{table}

\begin{table}[thb]
\centering
\caption{$\theta$ solved by (\ref{theta})}\label{theta-radius}
\setlength{\tabcolsep}{2mm}{
\resizebox{\textwidth}{10mm}{
\begin{tabular}{c|cccccccccccc}
\hline
$\bar{\gamma}/\gamma_1$ & 0.25 & 0.5 & 0.75 & 1 & 9 & 25 & 49 & 100 & 225 &	400 & 625 \\
\hline
Normal & 3.9144 & 3.5688 & 3.4250 & 3.0183 & 3.0422 & 3.0010 & 2.2386 & 0.7087 & 0.0209 & 9.05E-10 & 1.06E-09 \\
Box & 3.8231 & 3.4642 & 3.3186 & 2.9223 & 2.9460 & 2.9046 & 2.0963 & 0.6029 & 0.0113 & 8.87E-10 & 1.17E-09 \\
Ellipsoid & 3.6027 & 3.3241 & 3.2136 & 2.8214 & 2.8517 & 2.8846 & 2.9208 & 2.9753 & 2.4270 & 0.9938 & 0.2171 \\
\hline
\end{tabular}   }}
\end{table}

In \cref{normal,box,ellipsoid}, the means of optimal objective values are all have the same non-decreased trend as $\bar{\gamma}$ increases, which means the larger radius of the core set is, the more conservative GDRO models are. The same trend under three different distributions is trivial because both DRO and GDRO models are not relative to the true distribution. Therefore, we can use GDRO1.1 or GDRO2.1 if it is uncertain that the sample space of the true probability distribution is bounded or not.

It is obvious that GDRO1.1 and GDRO2.1 both give smaller optimal objective values than DRO1. Hence, GDRO models can reduce the degree of conservatism compared with DRO1. Meanwhile, GDRO1.1 and GDRO2.1 provide smaller variances of the optimal objective values than DRO2 in most cases. Thus, GDRO models are less affected by the set of samples than DRO models.

In addition, $\theta$ solved by \cref{theta} in \cref{theta-radius} is non-increased and tends to 0 as $\bar{\gamma}$ increases. Therefore, we can control the degree of conservatism in GDRO models by adjusting $\theta$ only when the radius of $Y$ is not extremely large. Besides, the variances of the optimal objective values when $\bar{\gamma}/\gamma_1$ is large ($\bar{\gamma}/\gamma_1 \geq 25$ in \cref{normal,box}, $\bar{\gamma}/\gamma_1 \geq 225$ in \cref{ellipsoid}) are obviously larger than those when $\bar{\gamma}/\gamma_1$ is small for all three models. Thus, GDRO models with smaller core set is more robust with respect to samples. In summary, we should only pay attention to the core data while constructing the core set $Y$.

\subsection{Selection of core set under bounded sample space}\label{bounded_value}
In this subsection, the 30 samples are got from the uniform distribution on the box $[0,20]^p$ with the small perturbation $N(\epsilon \bm{e}_p,0.5\bm{\mathcal{I}}_p)$ in the numerical experiments. According to the numerical results in \cref{theta-value}, we still fix $d_0=0.5$ in \cref{ambiguity2} and choose $\theta=0.5$ here.

\begin{table}[tbhp]
\centering
\caption{Changing the radius of the sample space}\label{bounded_outer}
\setlength{\tabcolsep}{3mm}{
\resizebox{\textwidth}{16mm}{
\begin{tabular}{c|c|ccccccccc}
\hline
\multicolumn{2}{c|}{$\gamma/\gamma_1$,$\bar{\gamma}_2/\gamma_1$} & 1 & 9 & 25 & 49 & 100 & 225 &	400 & 625 \\
\hline
\multirow{3}*{Mean} & DRO2   & -13.5297 & -13.5297 & -13.2687 & -11.5556 & -9.6003 & -9.3258 & -9.3221 & -9.3219 \\
~ & GDRO1.2 & -13.5297 & -13.5297 & -13.4770 & -13.0307 & -11.3912 & -10.7567 & -10.7358 & -10.7346 \\
~ & GDRO2.2 & -13.5297 & -13.5297 & -13.4927 & -13.1064 & -12.3795 & -11.8132 & -11.5382 & -11.3797 \\
\hline
\multirow{3}*{Var} & DRO2     & 0.2723 & 0.2723 & 2.1182 & 6.0838 & 4.0087 & 3.2526 & 3.2286 & 3.2270 \\
~ & GDRO1.2  & 0.2723 & 0.2723 & 0.5422 & 2.3658 & 2.7478 & 1.8478 & 1.7785 & 1.7733 \\
~ & GDRO2.2  & 0.2723 & 0.2723 & 0.3633 & 0.6044 & 0.5106 & 0.4000 & 0.3495 & 0.3276 \\
\hline
\end{tabular}  }}
\end{table}

\begin{table}[tbhp]
\centering
\caption{Changing the radius of the core set}\label{bounded_inner}
\setlength{\tabcolsep}{3mm}{
\resizebox{\textwidth}{25mm}{
\begin{tabular}{c|c|cccccccc}
\hline
\multicolumn{2}{c|}{$\gamma/\gamma_1$,$\bar{\gamma}_1/\gamma_1$} & 0.25 & 0.5 & 0.75 & 1 & 9 & 25 & 49 & 100  \\
\hline
\multirow{5}*{Mean} & DRO2    & -14.3210 & -13.9231 & -13.6490 & -13.4306 & -13.4306 & -13.2118 & -11.0446 & -9.2762 \\
~ & GDRO1.1 & -10.5824 & -10.5352 & -10.5159 & -10.4996 & -10.2376 & -9.9418 & -9.6131 & -9.1724 \\
~ & GDRO2.1 & -11.6007 & -11.2977 & -11.1412 & -11.1330 & -10.9953 & -10.6388 & -9.2100 & -9.0303 \\
~ & GDRO1.2 & -10.5919 & -10.5568 & -10.5364 & -10.5192 & -10.2476 & -9.9466 & -9.6154 & -9.1730 \\
~ & GDRO2.2 & -12.2025 & -11.8725 & -11.7011 & -11.6830 & -11.3672 & -10.7681 & -9.2130 & -9.0318 \\
\hline
\multirow{5}*{Var} & DRO2     & 0.1920 & 0.1869 & 0.1924 & 0.2013 & 0.2013 & 0.8199 & 5.4549 & 3.7923 \\
~ & GDRO1.1  & 1.5599 & 1.5212 & 1.5425 & 1.5608 & 1.8819 & 2.3088 & 2.8588 & 3.2493 \\
~ & GDRO2.1  & 0.2188 & 0.2498 & 0.2974 & 0.2978 & 0.3032 & 0.8085 & 2.9614 & 2.8832 \\
~ & GDRO1.2  & 1.6007 & 1.6028 & 1.6211 & 1.6370 & 1.9276 & 2.3347 & 2.8726 & 3.2532 \\
~ & GDRO2.2  & 0.2608 & 0.2922 & 0.3448 & 0.3466 & 0.3804 & 0.9548 & 2.9771 & 2.8928 \\
\hline
\end{tabular}
}}
\end{table}

In \cref{bounded_outer}, we fix the size $\bar\gamma_1$ of the core set $\mathcal{E}(\bar\gamma_1)$ and change the size $\bar\gamma_2$ of the sample space $\mathcal{E}(\bar\gamma_2)$  in GDRO1.2 and GDRO2.2. Besides, the radius of the sample space $\mathcal{E}(\gamma)$ of DRO2 is also changed accordingly. Let $\bar{\bm{\mu}} = \bm{\mu}_0$, $\bar{\gamma}_1 = \gamma_1$ and $\gamma$ and $\bar{\gamma}_2$ be 1, 9, 25, 49, 100, 225, 400, 625 times of $\gamma_1$, respectively. We repeat the experiments 100 times and compute the average and the variance of optimal objective values for each configuration.

Obviously, the optimal objective value increases with the increase of the radius of the sample space for all three models in \cref{bounded_outer}. However, GDRO models provide smaller optimal values and hence reduce the degree of conservatism significantly, compared with the DRO model with the same sample space. Besides, GDRO models also provide smaller variances, meaning more robust with respect to different sets of samples than the DRO model.

When the radius of the sample space is small ($\bar{\gamma}_2/\gamma_1 \leq 25$), there is little difference between the objective values of GDRO1.2 and GDRO2.2. The difference between GDRO1.2 and GDRO2.2 becomes more obvious with the increase of the radius. Generally, GDRO2.2 provides smaller optimal values and smaller variances and shows better performance than GDRO1.2. Probably the reason is that GDRO2.2 is influenced by $d_0$, which also controls the degree of conservatism and robustness in GDRO2.2.

In \cref{bounded_inner}, we fix the size of the sample space ($\mathbb{R}^p$ for GDRO1.1 and GDRO2.1, and  $\mathcal{E}(\bar\gamma_2)$ for GDRO1.2 and GDRO2.2) and change the radius ($\sqrt{\bar\gamma}$ for GDRO1.1 and GDRO2.1, and $\sqrt{\bar\gamma_1}$ for GDRO1.2 and GDRO2.2) of the core set in all four GDRO models. Besides, the radius of the sample space $\mathcal{E}(\gamma)$ of DRO2 is also changed accordingly. Let $\bar{\bm{\mu}} = \bm{\mu}_0$, $\bar{\gamma}_2 = 225\gamma_1$ and $\gamma$, $\bar{\gamma}$ and $\bar{\gamma}_1$ be 0.25,0.5,0.75,1,9,25,49,100 times of $\gamma_1$, respectively. We repeat the experiments 100 times and compute the average and the variance of optimal objective values.

Obviously, the optimal objective value increases with the increase of the radius of the core set for all four GDRO models in \cref{bounded_inner}. The DRO model is the least conservative because the sample space in the DRO model is the same as the core set in GDRO models. It can be seen that the bounded GDRO models (GDRO1.2 and GDRO2.2) are less conservative and more robust with respect to samples than the unbounded cases (GDRO1.1 and GDRO2.1), which shows that the sample space still play an important role in modelling. The variances of optimal objective values of GDRO models are worse than those of DRO2 because the radius of the sample space of GDRO models is larger than that of DRO2. Similar to that in \cref{radius_unbounded}, GDRO models with smaller core set is also more robust with respect to samples in the bounded case.

Therefore, a smaller core set and sample space are preferred in order to achieve a lower degree of conservatism.

\subsection{Data-based GDROs}\label{data_based}
Obviously, all sample data are bounded in real world. If a bounded sample space is selected in the DRO modelling, the probability distribution with unbounded sample space is excluded, such as the normal distribution. Therefore, GDRO is an efficient tool which not only includes the probability distribution with the whole space as the sample space, but also controls the degree of conservatism by constructing the core set and adjusting the parameter $\theta$ or $d_0$.

When an unbounded sample space is considered for our GDRO modelling, the determination of the core set of samples is needed first. The minimal ellipsoid containing all samples with the variance matrix of the samples is used in our modelling. The following example shows our procedure of handling the portfolio optimization model \cref{portfolio}. Here the normal distribution $N(10 \bm{e}_p,\frac{100}{3} \bm{\mathcal{I}}_p)$ is chosen as the true probability distribution and 30 samples are generated accordingly. For each sample, a small perturbation $N(\epsilon \bm{e}_p,0.5\bm{\mathcal{I}}_p)$ is added to the sample where $\epsilon \in (-0.5,0.5)$. The mean $\bm{\mu}_0$ and the covariance matrix $\bm{\Sigma}_0$ of the samples are thus computed. Then the smallest $\gamma$ satisfying that all samples are included in the ellipsoid $\{\bm{\xi} \in \mathbb{R}^p |(\bm{\xi}-\bm{\mu}_0)^{T}\bm{\Sigma}_0^{-1}(\bm{\xi}-\bm{\mu}_0) \leq \gamma \}$ is computed, denoted as $\tilde{\gamma}$. If $\theta=4$ and $d_0=0.5$ are fixed, then solve DRO1, DRO2, GDRO1.1 and GDRO2.1. The means and variances of optimal objective values are shown in \cref{ubd}.

\begin{table}[tbhp]
\centering
\caption{Unbounded case}\label{ubd}
\resizebox{\textwidth}{7mm}{
\begin{tabular}{cccccccc}
\hline
 & v(DRO1)& v(DRO2,$\tilde{\gamma}$) & v(DRO2,25$\tilde{\gamma}$) & v(GDRO1.1,$\tilde{\gamma}$) & v(GDRO1.1,$25\tilde{\gamma}$) & v(GDRO2.1,$\tilde{\gamma}$) & v(GDRO2.1,25$\tilde{\gamma}$) \\
\hline
Mean & -8.9815 & -13.4145 & -8.9822 & -13.4145 & -8.9822 & -13.0301 & -8.9816 \\
Var & 3.8781 & 0.2573 & 3.8838 & 0.2573 & 3.8838 & 0.2542 & 3.8786 \\
\hline
\end{tabular}}
\end{table}
Due to the big gap between $v(DRO1)$ and $v(DRO2,\tilde{\gamma})$, GDRO models make sense in application. Now there is a large alternative range for $\theta$ to control the degree of conservatism when the core set is set $\mathcal{E}(\tilde\gamma)=\{\bm{\xi} \in \mathbb{R}^p |(\bm{\xi}-\bm{\mu}_0)^{T}\bm{\Sigma}_0^{-1}(\bm{\xi}-\bm{\mu}_0) \leq \tilde\gamma \}$. Here $\theta=4$ in GDRO models is fixed for our convenience of repeating the experiments. Actually, decision makers can solve \cref{theta} to get a $\theta$ in line with their own risk preference. If the core set is selected as $\mathcal{E}(25\tilde{\gamma})$, then GDRO models are useless as the gap between $v(DRO1)$ and $v(DRO2,25\tilde{\gamma})$ is too small.

Suppose that it is known the sample space is bounded. Here the uniform distribution on the box $[0,20]^p$ is chosen as the true probability distribution and 30 samples with the small perturbation as above are generated accordingly. Then the smallest $\gamma$ satisfying that all samples are included in the ellipsoid $\{\bm{\xi} \in \mathbb{R}^p |(\bm{\xi}-\bm{\mu}_0)^{T}\bm{\Sigma}_0^{-1}(\bm{\xi}-\bm{\mu}_0) \leq \gamma \}$ is computed, denoted as $\tilde{\gamma}$. Let $\theta=0.5$ and $d_0=0.5$. The means and variances of optimal values are shown in \cref{bd}.

\begin{table}[ht]
\centering
\caption{Bounded case}\label{bd}
\resizebox{\textwidth}{7mm}{
\begin{tabular}{ccccccccc}
\hline
 & v(DRO1)& v(DRO2,$\tilde{\gamma}$) & v(DRO2,4$\tilde{\gamma}$) & v(DRO2,25$\tilde{\gamma}$) & v(GDRO1.2,$\tilde{\gamma}$,4$\tilde{\gamma}$) & v(GDRO1.2,$\tilde{\gamma}$,$25\tilde{\gamma}$) & v(GDRO2.2,$\tilde{\gamma}$,4$\tilde{\gamma}$) & v(GDRO2.2,$\tilde{\gamma}$,25$\tilde{\gamma}$) \\
\hline
Mean & -9.2564 & -13.5379 & -11.0322 & -9.2564 & -12.0367 & -10.3411 & -12.5736 & -11.2220 \\
Var & 2.8434 & 0.1745 & 5.7760 & 2.8434 & 3.6655 & 1.9666 & 1.0293 & 0.2975 \\
\hline
\end{tabular}}
\end{table}

We observe that $v(DRO2,25\tilde{\gamma})$ is extremely similar to $v(DRO1)$. It is why we add the case $4\tilde{\gamma}$ in Table \ref{bd}. For the two sample spaces of different sizes, the gap between $v(GDRO1.2,\tilde{\gamma},4\tilde{\gamma})$ and $v(GDRO1.2,\tilde{\gamma},25\tilde{\gamma})$ or $v(GDRO2.2,\tilde{\gamma},4\tilde{\gamma})$ and $v(GDRO2.2,\tilde{\gamma},25\tilde{\gamma})$ is different. Therefore, it is essential to estimate an appropriate sample space for the bounded case. If a bounded sample space cannot be estimated very well, the unbounded or a large bounded sample space like $\mathcal{E}(25\tilde{\gamma})$ with the GDRO modelling is suggested in application. In fact, there is little difference between the results of bounded and unbounded cases if the bounded sample space is large enough.

In summary, the decision maker can determine the core set $Y$ and the sample space $\Xi$ firstly, then compute the optimal values of DRO with $Y$ and $\Xi$ as the sample space, respectively. If the gap between the two values is extremely small, there is no need to use the GDRO models. Conversely, GDRO models are good choice in controlling the degree of conservatism.

\section{Conclusions}
In this paper, we introduce the globalized distributionally robust optimization and derive the corresponding tractable GDRO models under the assumption that the objective function is piecewise linear. Compared with DRO, the GDRO approach allows wider range of the sample space and controls the degree of conservatism at the same time. Besides, the GDRO approach is more flexible in adjustment of conservatism by adjusting the parameter $\theta$.
In application, the decision maker can determine the core set $Y$ and the sample space $\Xi$ firstly. If there is no much priori information of sample space, the unbounded case $\Xi=\mathbb{R}^p$ is suggested. Then the optimal objective values of DRO with $Y$ and $\Xi$ as the sample space can be computed respectively. If the gap between the two values is extremely small, DRO may be enough to get a satisfying result. If the gap is large, the decision maker can determine an appropriate $\theta$ by \cref{theta} and solve the corresponding GDRO model.

\appendix
\section{Proofs}
\subsection{Proof of \texorpdfstring{\cref{dual}}{lemma2.1}}
\begin{proof}
We consider the inner maximization problem of \cref{DRO} with the ambiguity set \cref{ambiguity} and rewrite it as
\begin{equation}\label{P}
\begin{split}
\max \limits_{P \in C(\mathcal{P})} \quad & E_{P}[h(\bm{x},\bm{\xi})]  \\
s.t. \quad & E_{P}[(\bm{\xi}-\bm{\mu}_0)(\bm{\xi}-\bm{\mu}_0)^{T}] \preceq \gamma_2\bm{\Sigma}_0, \\
     & E_{P}\begin{bmatrix}
                           -\bm{\Sigma}_0 & (\bm{\mu_0}-\bm{\xi}) \\
                           (\bm{\mu_0}-\bm{\xi})^{T} & -\gamma_1 \\
                           \end{bmatrix}  \preceq 0, \\
     & E_{P}[\mathbb{I}_{\bm{\xi} \in \Xi}] = 1,
\end{split}
\end{equation}
where $C(\mathcal{P})$ is the cone generated by $\mathcal{P}$ and $\mathbb{I}_{\bm{\xi} \in \Xi}$ is the indicator function of the set $\Xi$, defined as
\begin{equation*}
   \mathbb{I}_{\bm{\xi} \in \Xi} (\bm{\xi})=
   \left\{ \begin{aligned}  1 \quad &\mathrm{if} \quad \bm{\xi} \in \Xi, \\  0 \quad & \mathrm{otherwise}. \end{aligned}\right.
\end{equation*}

For the given $\bm{\Lambda}_1 \in \mathbb{S}_+^p$, $\bigl[ \begin{smallmatrix} \bm{\Lambda}_2 & \bm{\beta} \\ \bm{\beta}^{T} & s \\ \end{smallmatrix} \bigr] \in \mathbb{S}_+^{p+1}$ and $t \in \mathbb{R}$, the Lagrangian function of \cref{P} is
\begin{equation*}
\begin{aligned}
&L \left( \bm{\Lambda}_1,\begin{bmatrix} \bm{\Lambda}_2 & \bm{\beta} \\ \bm{\beta}^T & s \\ \end{bmatrix},t,P \right)\\
= \quad & E_{P}[h(\bm{x},\bm{\xi})]  + \bm{\Lambda}_1 \cdot (\gamma_2 \bm{\Sigma}_0 - E_{P}[(\bm{\xi}-\bm{\mu}_0)(\bm{\xi}-\bm{\mu}_0)^{T}]) \\
&  + \begin{bmatrix} \bm{\Lambda}_2 & \bm{\beta} \\ \bm{\beta}^T & s \\  \end{bmatrix}  \cdot E_{P} \begin{bmatrix} \bm{\Sigma}_0 & (\bm{\xi}-\bm{\mu_0}) \\ (\bm{\xi}-\bm{\mu_0})^{T} & \gamma_1 \\ \end{bmatrix}  + t(1-E_{P}[\mathbb{I}_{\bm{\xi} \in \Xi}]) \\
= \quad & \bm{\Lambda}_1 \cdot (\gamma_2 \bm{\Sigma}_0 -\bm{\mu}_0 \bm{\mu}_0^{T})+\bm{\Lambda}_2 \cdot \bm{\Sigma}_0+ \gamma_1 s - 2\bm{\mu}_0^{T} \bm{\beta} + t \\
&  + E_P[h(\bm{x},\bm{\xi}) - \bm{\xi}^{T}\bm{\Lambda}_1 \bm{\xi} + 2\bm{\xi}^{T}\bm{\Lambda}_1 \bm{\mu_0} + 2\bm{\beta}^{T}\bm{\xi} - t \mathbb{I}_{\bm{\xi} \in \Xi}]. \\
\end{aligned}
\end{equation*}

And
\begin{equation*}
\begin{aligned}
&\max \limits_{P \in C(\mathcal{P})}  L \left(\bm{\Lambda}_1,\begin{bmatrix} \bm{\Lambda}_2 & \bm{\beta} \\ \bm{\beta}^T & s \\ \end{bmatrix},t,P \right)\\
= \quad & \bm{\Lambda}_1 \cdot (\gamma_2 \bm{\Sigma}_0 -\bm{\mu}_0 \bm{\mu}_0^{T})+\bm{\Lambda}_2 \cdot \bm{\Sigma}_0+ \gamma_1 s - 2\bm{\mu}_0^{T} \bm{\beta} + t \\
& + \max \limits_{P \in C(\mathcal{P})} E_P[h(\bm{x},\bm{\xi}) - \bm{\xi}^{T}\bm{\Lambda}_1 \bm{\xi} + 2\bm{\xi}^{T}\bm{\Lambda}_1 \bm{\mu_0} + 2\bm{\beta}^{T}\bm{\xi} - t \mathbb{I}_{\bm{\xi} \in \Xi}] \\
= \quad & \left\{ \begin{aligned} & \bm{\Lambda}_1 \cdot (\gamma_2 \bm{\Sigma}_0 -\bm{\mu}_0 \bm{\mu}_0^{T})+\bm{\Lambda}_2 \cdot \bm{\Sigma}_0+ \gamma_1 s - 2\bm{\mu}_0^{T} \bm{\beta} + t \\
& \qquad  \text{if} \quad  E_P[h(\bm{x},\bm{\xi}) - \bm{\xi}^{T}\bm{\Lambda}_1 \bm{\xi} + 2\bm{\xi}^{T}\bm{\Lambda}_1 \bm{\mu_0} + 2\bm{\beta}^{T}\bm{\xi} - t \mathbb{I}_{\bm{\xi} \in \Xi}] \leq 0 \qquad \forall P \in C(\mathcal{P}), \\
& \infty \quad \text{otherwise},
\end{aligned}\right.
\end{aligned}
\end{equation*}
where the constraint
\begin{equation*}
\begin{split}
& E_P[h(\bm{x},\bm{\xi}) - \bm{\xi}^{T}\bm{\Lambda}_1 \bm{\xi} + 2\bm{\xi}^{T}\bm{\Lambda}_1 \bm{\mu_0} + 2\bm{\beta}^{T}\bm{\xi} - t \mathds{1}_{\bm{\xi} \in \Xi}] \leq 0 \qquad \forall P \in C(\mathcal{P}) \\
\Longleftrightarrow \quad
& E_P[h(\bm{x},\bm{\xi}) - \bm{\xi}^{T}\bm{\Lambda}_1 \bm{\xi} + 2\bm{\xi}^{T}\bm{\Lambda}_1 \bm{\mu_0} + 2\bm{\beta}^{T}\bm{\xi} - t \mathds{1}_{\bm{\xi} \in \Xi}] \leq 0 \qquad \forall P \in \mathcal{P} \\
\Longleftrightarrow \quad
& h(\bm{x},\bm{\xi})-\bm{\xi}^{T}\bm{\Lambda}_1 \bm{\xi}+2\bm{\xi}^{T}\bm{\Lambda}_1 \bm{\mu_0}+2\bm{\beta}^{T}\bm{\xi} \leq t \qquad \forall \bm{\xi} \in \Xi.
\end{split}
\end{equation*}

Therefore, the Lagrangian dual of \cref{P} is written as
\begin{equation}\label{D}
\begin{split}
\min \limits_{\bm{\Lambda}_1,\bm{\Lambda}_2,\bm{\beta},s,t} \quad & \bm{\Lambda}_1 \cdot (\gamma_2 \bm{\Sigma}_0 -\bm{\mu}_0 \bm{\mu}_0^{T})+\bm{\Lambda}_2 \cdot \bm{\Sigma}_0+ \gamma_1 s - 2\bm{\mu}_0^{T} \bm{\beta} + t \\
s.t. \quad & h(\bm{x},\bm{\xi})-\bm{\xi}^{T}\bm{\Lambda}_1 \bm{\xi}+2\bm{\xi}^{T}\bm{\Lambda}_1 \bm{\mu_0}+2\bm{\beta}^{T}\bm{\xi} \leq t \qquad \forall \bm{\xi} \in \Xi, \\
& \bm{\Lambda}_1 \in \mathbb{S}_+^p, \quad \begin{bmatrix} \bm{\Lambda}_2 & \bm{\beta} \\ \bm{\beta}^{T} & s \\ \end{bmatrix}  \in \mathbb{S}_+^{p+1}, \quad t \in \mathbb{R}.
\end{split}
\end{equation}

We can simplify the dual problem by solving analytically for the variables $(\bm{\Lambda}_2,\bm{\beta},s)$ with fixed $(\bm{\Lambda}_1,t)$. Denote $(\bm{\Lambda}^*_2,\bm{\beta}^*,s^*)$ as the optimal solution of \cref{D} with fixed $(\bm{\Lambda}_1,t)$.

\begin{itemize}
\item If $s^*>0$, then $\bm{\Lambda}^*_2 \succeq \frac{1}{s^*} \bm{\beta}^* \bm{\beta}^{*T}$ by $\bigl[ \begin{smallmatrix} \bm{\Lambda}^*_2 & \bm{\beta}^* \\ \bm{\beta}^{*T} & s^* \\ \end{smallmatrix} \bigr] \in \mathbb{S}_+^{p+1}$ according to Schur's complement \cite{conv_opt_book}.

The objective function gets its minimum at $\bm{\Lambda}_2^* = \frac{1}{s^*} \bm{\beta}^* \bm{\beta}^{*T}$. Then the objective function becomes
$$t+\bm{\Lambda}_1 \cdot (\gamma_2 \bm{\Sigma}_0 -\bm{\mu}_0 \bm{\mu}_0^{T})+\frac{1}{s^*} \bm{\beta}^{*T} \bm{\Sigma}_0 \bm{\beta}^*+ \gamma_1 s^* - 2\bm{\mu}_0^{T}\bm{\beta}^*.$$

Since $\frac{1}{s^*} \bm{\beta}^{*T} \bm{\Sigma}_0 \bm{\beta}^*+ \gamma_1 s^* \geq 2\sqrt{\gamma_1 \bm{\beta}^{*T} \bm{\Sigma}_0 \bm{\beta}^*}$ and the equality holds if and only if $s^*=\sqrt{\bm{\beta}^{*T} \bm{\Sigma}_0 \bm{\beta}^*/\gamma_1}$, the objective function becomes
$$t+\bm{\Lambda}_1 \cdot (\gamma_2 \bm{\Sigma}_0 +\bm{\mu}_0 \bm{\mu}_0^{T})+ \sqrt{\gamma_1}||\bm{\Sigma}_0^{1/2} (\bm{q}+2\bm{\Lambda}_1\bm{\mu}_0)||_2 + \bm{q}^{T} \bm{\mu}_0$$
with $\bm{q}=-2(\bm{\beta}^*+\bm{\Lambda}_1 \bm{\mu}_0)$.
\item If $s^*=0$, then $\bm{\beta}^*=0$ to keep $\bigl[ \begin{smallmatrix} \bm{\Lambda}^*_2 & \bm{\beta}^* \\ \bm{\beta}^{*T} & s^* \\ \end{smallmatrix} \bigr]   \in \mathbb{S}_+^{p+1}$.

The objective function gets its minimum at $\bm{\Lambda}_2^*=0$ due to $\bm{\Sigma}_0 \succeq 0$.

Let $\bm{q}=-2(\bm{\beta}^*+\bm{\Lambda}_1 \bm{\mu}_0)$. The objective function is written as
\begin{equation*}
\begin{split}
& t+\bm{\Lambda}_1 \cdot (\gamma_2 \bm{\Sigma}_0 -\bm{\mu}_0 \bm{\mu}_0^{T}) \\
= \quad & t+\bm{\Lambda}_1 \cdot (\gamma_2 \bm{\Sigma}_0 +\bm{\mu}_0 \bm{\mu}_0^{T})+ \sqrt{\gamma_1}||\bm{\Sigma}_0^{1/2} (\bm{q}+2\bm{\Lambda}_1\bm{\mu}_0)||_2 + \bm{q}^{T} \bm{\mu}_0.
\end{split}
\end{equation*}
\end{itemize}

Thus, the dual of \cref{P} can be equivalently rewritten as
\begin{equation}\label{D2}
\begin{split}
\min \limits_{\bm{\Lambda},\bm{q},t} \quad & t+\bm{\Lambda} \cdot (\gamma_2 \bm{\Sigma}_0 +\bm{\mu}_0 \bm{\mu}_0^{T})+ \sqrt{\gamma_1}||\bm{\Sigma}_0^{1/2} (\bm{q}+2\bm{\Lambda} \bm{\mu}_0)||_2 + \bm{q}^{T} \bm{\mu}_0 \\
s.t. \quad & h(\bm{x},\bm{\xi})-\bm{\xi}^{T}\bm{\Lambda} \bm{\xi} - \bm{q}^{T} \bm{\xi} \leq t \qquad \forall \bm{\xi} \in \Xi, \\
& \bm{\Lambda} \in \mathbb{S}_+^p, \quad \bm{q} \in \mathbb{R}^p, \quad t \in \mathbb{R}.
\end{split}
\end{equation}

According to Proposition 2.8 in Shapiro\cite{Shapiro}, we can conclude that there is no dual gap if $\gamma_1,\gamma_2>0$, $\bm{\Sigma}_0 \succ 0$ and $h(\bm{x},\bm{\xi})$ is integrable for all $P\in \mathcal{D}$. Besides, the optimal solution set of the problem \cref{D2} is nonempty.
\end{proof}

\subsection{Proofs of \texorpdfstring{\cref{thm1,thm2,thm3,thm4}}{thm1 to thm4}}

We first introduce the following theorem given in \cite{globalized_ro}, which is essential to our computationally tractable GDRO models.
\begin{theorem}[Ben-Tal et al. \cite{globalized_ro}]
\label{bental}
Let $f(\cdot,\bm{x})$ be a closed proper concave function in $\mathbb{R}^m$ for all $\bm{x} \in \mathbb{R}^n$ and
$\phi:\mathbb{R}^m \times \mathbb{R}^m \to \mathbb{R}$ be a closed, jointly convex and nonnegative function for which $\phi(\bm{\xi},\bm{\xi})=0$ for all $\bm{\xi} \in \mathbb{R}^m$. Let the set $Z_1 \subset \mathbb{R}^L$ be nonempty, convex and compact with $0 \in ri(Z_1)$ and $Z_2$ be a closed convex set such that $Z_1 \subset Z_2$; and let $U_1$ and $U_2$ be defined as follows for fixed $\bm{\xi}_0$ and $\bm{A} \in \mathbb{R}^{m \times L}$ such that $\bm{\xi}_0 \in ri(U_1)$.

\begin{equation}
\label{uncertaintyset}
U_i = \{ \bm{\xi} = \bm{\xi}_0 + \bm{A} \bm{\zeta} | \bm{\zeta} \in Z_i\}, \qquad i=1,2.
\end{equation}

Then $\bm{x}$ satisfies
\begin{equation}
\label{theo_1}
f(\bm{\xi},\bm{x}) \leq \min \limits_{\bm{\xi}' \in U_1} \phi(\bm{\xi},\bm{\xi}') \qquad \forall \bm{\xi} \in U_2,
\end{equation}
if and only if there exist $\bm{v},\bm{w} \in \mathbb{R}^m$ that satisfy the single inequality
\begin{equation}
\label{theo_2}
\bm{\xi}_0^{T}(\bm{v}+\bm{w})+\delta^*(\bm{A}^{T}\bm{v}|Z_1) + \delta^*(\bm{A}^{T}\bm{w}|Z_2)-f_*(\bm{v}+\bm{w},\bm{x})+\phi^{**}(\bm{v},-\bm{v}) \leq 0.
\end{equation}
\end{theorem}

\emph{Notation.} For any function $g:\mathbb{R}^n \to \mathbb{R}$, let $\mathrm{dom}(g)=\{\bm{x} \in \mathbb{R}^n | g(\bm{x}) < \infty \}$. The convex conjugate of $g$ is defined as
$$g^*(\bm{y})=\sup \limits_{\bm{x}\in \mathrm{dom}(g)} \{\bm{y}^{T} \bm{x} -g(\bm{x})\}.$$
The concave conjugate function of $g$ is defined as
$$g_*(\bm{y})=\inf \limits_{\bm{x}\in \mathrm{dom}(-g)} \{\bm{y}^{T} \bm{x} -g(\bm{x})\}.$$
For a function $f(\cdot,\cdot)$ with two variables, let $f^*(\cdot,\cdot)$ and $f_*(\cdot,\cdot)$ denote the convex and concave conjugate function with respect to the first variable, respectively, while $f^{**}(\cdot,\cdot)$ and $f_{**}(\cdot,\cdot)$ denote the convex and concave conjugate function with respect to both variables, respectively.

The indicator function on the set $S$ is defined as
\begin{equation*}
   \delta(\bm{x}|S)=
   \left\{ \begin{aligned}  0 \qquad &\mathrm{if} \quad \bm{x} \in S, \\  \infty \qquad & \mathrm{otherwise}. \end{aligned}\right.
\end{equation*}

\begin{proof}[Proof of \cref{thm1}]
Note that
$$h(\bm{x},\bm{\xi})=\max \limits_{k} \{\bm{a}_k(\bm{x})^{T} \bm{\xi}+b_k(\bm{x})\}, \qquad k \in \{1,2,...,K\}.$$
So the constraint
\begin{equation*}
h(\bm{x},\bm{\xi})-\bm{\xi}^{T}\bm{\Lambda} \bm{\xi} - \bm{q}^{T} \bm{\xi} - t \leq \min \limits_{\bm{\xi}' \in Y} \phi(\bm{\xi},\bm{\xi}') \qquad \forall \bm{\xi} \in \mathbb{R}^p
\end{equation*}
is equivalent to
\begin{equation*}
\bm{a}_k(\bm{x})^{T} \bm{\xi}+b_k(\bm{x})-\bm{\xi}^{T}\bm{\Lambda} \bm{\xi} - \bm{q}^{T} \bm{\xi} - t \leq \min \limits_{\bm{\xi}' \in Y} \phi(\bm{\xi},\bm{\xi}') \quad \forall \bm{\xi} \in \mathbb{R}^p \quad \forall k.
\end{equation*}
Let
\begin{equation*}
\begin{split}
& m = p, \\
& Z_1 = \{\bm{\zeta} | \|\bm{\zeta}\|_2 \leq \sqrt{\bar{\gamma}}\},  \qquad Z_2 = \mathbb{R}^p, \\
& \bm{A} = \bm{\Sigma}_0^{1/2}, \qquad \bm{\xi}_0 = \bar{\bm{\mu}}, \\
& U_1= \{\bm{\xi} | (\bm{\xi}-\bar{\bm{\mu}})^T \bm{\Sigma}_0^{-1} (\bm{\xi}-\bar{\bm{\mu}}) \leq \bar{\gamma}\} = Y, \qquad U_2 = \mathbb{R}^p, \\
& f_k(\bm{\xi},(\bm{x},\bm{\Lambda},\bm{q},t))=\bm{a}_k(\bm{x})^{T} \bm{\xi}+b_k(\bm{x})-\bm{\xi}^{T}\bm{\Lambda} \bm{\xi} - \bm{q}^{T} \bm{\xi} - t,\\
& \phi(\bm{\xi},\bm{\xi}') = \theta \| \bm{\xi}-\bm{\xi}'\|_2
\end{split}
\end{equation*}
in \cref{bental}.
Then
\begin{equation*}
\begin{split}
& \delta^*(\bm{A}^{T}\bm{v}|Z_1) = \sup \limits_{\|\bm{\zeta}\|_2 \leq \sqrt{\bar{\gamma}}} \bm{\zeta}^{T} \bm{A}^{T}\bm{v}
                                    = \sqrt{\bar{\gamma}}\| \bm{A}^{T}\bm{v} \|_2
                                    = \sqrt{\bar{\gamma}} \|\bm{\Sigma}_0^{1/2}\bm{v} \|_2,\\
& \delta^*(\bm{A}^{T}\bm{w}|Z_2) = \sup \limits_{\bm{\zeta} \in \mathbb{R}^p} \bm{\zeta}^{T}\bm{A}^{T}\bm{w}
                                    = \sup \limits_{\bm{y} \in \mathbb{R}^p} \bm{y}^{T} \bm{w}
                                    = \left\{ \begin{split}  0 \qquad &\mathrm{if} \quad \bm{w}=0, \\  \infty \qquad & \mathrm{otherwise}, \end{split}\right.
\end{split}
\end{equation*}
\begin{equation*}
 \begin{split}
 \phi^{**}(\bm{v},-\bm{v}) & = \sup \limits_{\bm{s},\bm{t}} \{(\bm{s}-\bm{t})^{T}\bm{v} - \theta \|\bm{s}-\bm{t}\|_2\}  \\
  & = \sup \limits_{\bm{s}} \{\bm{s}^{T}\bm{v} - \theta \|\bm{s}\|_2\} \\
  & = \left\{ \begin{split} 0 \qquad & \mathrm{if} \quad \|\bm{v}\|_2 \leq \theta, \\ \infty \qquad & \mathrm{otherwise}, \end{split} \right.
 \end{split}
\end{equation*}
\begin{equation*}
\begin{split}
f_{k,*}(\bm{y},(\bm{x},\bm{\Lambda},\bm{q},t))
& = \inf \limits_{\bm{\xi} \in \mathbb{R}^p} \{\bm{y}^{T} \bm{\xi}-f_k(\bm{\xi},(\bm{x},\bm{\Lambda},\bm{q},t))\} \\
& = \inf \limits_{\bm{\xi} \in \mathbb{R}^p} \{\bm{\xi}^{T} \bm{\Lambda} \bm{\xi} + (\bm{y}+\bm{q}-\bm{a}_k(\bm{x}))^{T}\bm{\xi} + t -b_k(\bm{x}) \}.
\end{split}
\end{equation*}

According to \cref{bental}, we know that
\begin{equation*}
\bm{a}_k(\bm{x})^{T} \bm{\xi}+b_k(\bm{x})-\bm{\xi}^{T}\bm{\Lambda} \bm{\xi} - \bm{q}^{T} \bm{\xi} - t \leq \min \limits_{\bm{\xi}' \in Y} \phi(\bm{\xi},\bm{\xi}') \qquad \forall \bm{\xi} \in \mathbb{R}^p
\end{equation*}
is equivalent to
\begin{equation*}
\left\{
\begin{aligned}
& \bar{\bm{\mu}}^{T} \bm{v}_k + \sqrt{\bar{\gamma}} \|\bm{\Sigma}_0^{1/2}\bm{v}_k  \|_2 -\inf \limits_{\bm{\xi} \in \mathbb{R}^p} \{\bm{\xi}^{T} \bm{\Lambda} \bm{\xi} + (\bm{v}_k +\bm{q}-\bm{a}_k(\bm{x}))^{T}\bm{\xi} + t -b_k(\bm{x}) \} \leq 0, \\
& \|\bm{v}_k \| \leq \theta,
\end{aligned}
\right.
\end{equation*}
where
%{\small
\begin{equation*}
\begin{split}
& \bar{\bm{\mu}}^{T} \bm{v}_k  + \sqrt{\bar{\gamma}} \|\bm{\Sigma}_0^{1/2}\bm{v}_k  \|_2 -\inf \limits_{\bm{\xi} \in \mathbb{R}^p} \{\bm{\xi}^{T} \bm{\Lambda} \bm{\xi} + (\bm{v}_k +\bm{q}-\bm{a}_k(\bm{x}))^{T}\bm{\xi} + t -b_k(\bm{x}) \} \leq 0 \\
\Longleftrightarrow \quad
& \bar{\bm{\mu}}^{T} \bm{v}_k  + \sqrt{\bar{\gamma}} \|\bm{\Sigma}_0^{1/2}\bm{v}_k  \|_2 \leq \bm{\xi}^{T} \bm{\Lambda} \bm{\xi} + (\bm{v}_k +\bm{q}-\bm{a}_k(\bm{x}))^{T}\bm{\xi} + t -b_k(\bm{x}) \qquad\forall \bm{\xi} \in \mathbb{R}^p\\
\Longleftrightarrow \quad
& \bm{\xi}^{T} \bm{\Lambda} \bm{\xi} + (\bm{v}_k +\bm{q}-\bm{a}_k(\bm{x}))^{T}\bm{\xi} + t -b_k(\bm{x})-\bar{\bm{\mu}}^{T} \bm{v}_k  - \sqrt{\bar{\gamma}} \|\bm{\Sigma}_0^{1/2}\bm{v}_k  \|_2 \geq 0 \qquad \forall \bm{\xi} \in \mathbb{R}^p \\
\Longleftrightarrow \quad
&[\bm{\xi}^{T},1] \begin{bmatrix}
      \bm{\Lambda} & \dfrac{1}{2}(\bm{v}_k +\bm{q}-\bm{a}_k(\bm{x})) \\
      \dfrac{1}{2}(\bm{v}_k +\bm{q}-\bm{a}_k(\bm{x}))^{T} & t-b_k(\bm{x})-\bar{\bm{\mu}}^{T} \bm{v}_k -\sqrt{\bar{\gamma}} \|\bm{\Sigma}_0^{1/2}\bm{v}_k  \|_2 \\
      \end{bmatrix}
      \begin{bmatrix}
      \bm{\xi} \\
      1\\
      \end{bmatrix}  \geq 0 \qquad \forall \bm{\xi} \in \mathbb{R}^p \\
\Longleftrightarrow \quad
& \begin{bmatrix}
      \bm{\Lambda} & \dfrac{1}{2}(\bm{v}_k +\bm{q}-\bm{a}_k(\bm{x})) \\
      \dfrac{1}{2}(\bm{v}_k +\bm{q}-\bm{a}_k(\bm{x}))^{T} & t-b_k(\bm{x})-\bar{\bm{\mu}}^{T} \bm{v}_k -\sqrt{\bar{\gamma}} \|\bm{\Sigma}_0^{1/2}\bm{v}_k  \|_2 \\
      \end{bmatrix}  \succeq 0.
\end{split}
\end{equation*}
%}
Then the result follows.
\end{proof}

\begin{proof}[Proof of \cref{thm2}]
Note that
$$h(\bm{x},\bm{\xi})=\max \limits_{k} \{\bm{a}_k(\bm{x})^{T} \bm{\xi}+b_k(\bm{x})\}, \qquad k \in \{1,2,...,K\}.$$
So the constraint
\begin{equation*}
h(\bm{x},\bm{\xi})-\bm{\xi}^{T}\bm{\Lambda} \bm{\xi} - \bm{q}^{T} \bm{\xi} - t \leq \min \limits_{\bm{\xi}' \in Y} \phi(\bm{\xi},\bm{\xi}') \qquad \forall \bm{\xi} \in \Xi
\end{equation*}
is equivalent to
\begin{equation*}
\bm{a}_k(\bm{x})^{T} \bm{\xi}+b_k(\bm{x})-\bm{\xi}^{T}\bm{\Lambda} \bm{\xi} - \bm{q}^{T} \bm{\xi} - t \leq \min \limits_{\bm{\xi}' \in Y} \phi(\bm{\xi},\bm{\xi}') \qquad \forall \bm{\xi} \in \Xi  \forall k.
\end{equation*}

Let
\begin{equation*}
\begin{split}
& m = p, \\
& Z_1 = \{\bm{\zeta} | \|\bm{\zeta}\|_2 \leq \sqrt{\bar{\gamma_1}}\},
 \qquad Z_2 = \{\bm{\zeta} | \|\bm{\zeta}\|_2 \leq \sqrt{\bar{\gamma_2}}\},  \\
& \bm{A} = \bm{\Sigma}_0^{1/2}, \qquad \bm{\xi}_0 = \bar{\bm{\mu}}, \\
& U_1= \{\bm{\xi} | (\bm{\xi}-\bar{\bm{\mu}})^T \bm{\Sigma}_0^{-1} (\bm{\xi}-\bar{\bm{\mu}}) \leq \bar{\gamma}_1\} = Y, \\
& U_2= \{\bm{\xi} | (\bm{\xi}-\bar{\bm{\mu}})^T \bm{\Sigma}_0^{-1} (\bm{\xi}-\bar{\bm{\mu}}) \leq \bar{\gamma}_2\} = \Xi, \\
& f_k(\bm{\xi},(\bm{x},\bm{\Lambda},\bm{q},t))=\bm{a}_k(\bm{x})^{T} \bm{\xi}+b_k(\bm{x})-\bm{\xi}^{T}\bm{\Lambda} \bm{\xi} - \bm{q}^{T} \bm{\xi} - t,\\
& \phi(\bm{\xi},\bm{\xi}') = \theta \| \bm{\xi}-\bm{\xi}'\|_2
\end{split}
\end{equation*}
in \cref{bental}.

Then
\begin{equation*}
\begin{split}
& \delta^*(\bm{A}^{T}\bm{v}|Z_1) = \sup \limits_{\|\bm{\zeta}\|_2 \leq \sqrt{\bar{\gamma_1}}} \bm{\zeta}^{T}\bm{A}^{T}\bm{v}
                                    = \sqrt{\bar{\gamma_1}} \| \bm{A}^{T}\bm{v} \|_2
                                    = \sqrt{\bar{\gamma_1}} \|\bm{\Sigma}_0^{1/2}\bm{v} \|_2,\\
& \delta^*(\bm{A}^{T}\bm{w}|Z_2) = \sup \limits_{\|\bm{\zeta}\|_2 \leq \sqrt{\bar{\gamma_2}}} \bm{\zeta}^{T}\bm{A}^{T}\bm{w}
                                    = \sqrt{\bar{\gamma_2}} \| \bm{A}^{T}\bm{w} \|_2
                                    = \sqrt{\bar{\gamma_2}} \|\bm{\Sigma}_0^{1/2}\bm{w} \|_2,\\
\end{split}
\end{equation*}
\begin{equation*}
 \begin{split}
 \phi^{**}(\bm{v},-\bm{v}) & = \sup \limits_{\bm{s},\bm{t}} \{(\bm{s}-\bm{t})^{T}\bm{v} - \theta \|\bm{s}-\bm{t}\|_2\}  \\
  & = \sup \limits_{\bm{s}} \{\bm{s}^{T}\bm{v} - \theta \|\bm{s}\|_2\} \\
  & = \left\{ \begin{split} 0 \qquad & \mathrm{if} \quad \|\bm{v}\|_2 \leq \theta, \\ \infty \qquad & \mathrm{otherwise}, \end{split} \right.
 \end{split}
\end{equation*}
\begin{equation*}
\begin{split}
f_{k,*}(\bm{y},(\bm{x},\bm{\Lambda},\bm{q},t))
& = \inf \limits_{\bm{\xi} \in \mathbb{R}^p} \{\bm{y}^{T} \bm{\xi}-f_k(\bm{\xi},(\bm{x},\bm{\Lambda},\bm{q},t))\} \\
& = \inf \limits_{\bm{\xi} \in \mathbb{R}^p} \{\bm{\xi}^{T} \bm{\Lambda} \bm{\xi} + (\bm{y}+\bm{q}-\bm{a}_k(\bm{x}))^{T}\bm{\xi} + t -b_k(\bm{x}) \}.
\end{split}
\end{equation*}

According to \cref{bental}, we know that
\begin{equation*}
\bm{a}_k(\bm{x})^{T} \bm{\xi}+b_k(\bm{x})-\bm{\xi}^{T}\bm{\Lambda} \bm{\xi} - \bm{q}^{T} \bm{\xi} - t \leq \min \limits_{\bm{\xi}' \in Y_1} \phi(\bm{\xi},\bm{\xi}') \qquad \forall \bm{\xi} \in Y_2
\end{equation*}
is equivalent to
%{\small
\begin{equation*}
\left\{
\begin{aligned}
& \bar{\bm{\mu}}^{T} (\bm{v}_k + \bm{w}_k) + \sqrt{\bar{\gamma_1}} \|\bm{\Sigma}_0^{1/2}\bm{v}_k  \|_2 +\sqrt{\bar{\gamma_2}} \|\bm{\Sigma}_0^{1/2}\bm{w}_k \|_2 -f_{k,*}(\bm{v}_k+\bm{w}_k,(\bm{x},\bm{\Lambda},\bm{q},t)) \} \leq 0, \\
& \|\bm{v}_k \|_2 \leq \theta,
\end{aligned}
\right.
\end{equation*}
%}
where
%{\small
\begin{equation*}
\begin{aligned}
& \bar{\bm{\mu}}^{T} (\bm{v}_k + \bm{w}_k) + \sqrt{\bar{\gamma_1}} \|\bm{\Sigma}_0^{1/2}\bm{v}_k  \|_2 +\sqrt{\bar{\gamma_2}} \|\bm{\Sigma}_0^{1/2}\bm{w}_k \|_2 - f_{k,*}(\bm{v}_k+\bm{w}_k,(\bm{x},\bm{\Lambda},\bm{q},t)) \} \leq 0 \\
~\\
\Longleftrightarrow \quad
& \bar{\bm{\mu}}^{T} (\bm{v}_k+\bm{w}_k)  + \sqrt{\bar{\gamma_1}} \|\bm{\Sigma}_0^{1/2}\bm{v}_k \|_2 + \sqrt{\bar{\gamma_2}} \|\bm{\Sigma}_0^{1/2}\bm{w}_k \|_2  \\
& \leq \bm{\xi}^{T} \bm{\Lambda} \bm{\xi} + (\bm{v}_k + \bm{w}_k +\bm{q}-\bm{a}_k(\bm{x}))^{T}\bm{\xi} + t -b_k(\bm{x}) \qquad \forall \bm{\xi} \in \mathbb{R}^p\\
~\\
\Longleftrightarrow \quad
&   \begin{bmatrix}
      \bm{\Lambda} & \dfrac{1}{2} (\bm{v}_k + \bm{w}_k +\bm{q}-\bm{a}_k(\bm{x}))\\[1.5em]
      \dfrac{1}{2}(\bm{v}_k + \bm{w}_k +\bm{q}-\bm{a}_k(\bm{x}))^{T} & \begin{aligned}t-b_k(\bm{x})-\bar{\bm{\mu}}^{T} (\bm{v}_k + \bm{w}_k) \\ -\sqrt{\bar{\gamma_1}} \|\bm{\Sigma}_0^{1/2}\bm{v}_k \|_2 - \sqrt{\bar{\gamma_2}} \|\bm{\Sigma}_0^{1/2}\bm{w}_k \|_2 \end{aligned}\\
      \end{bmatrix}  \succeq 0
\end{aligned}
\end{equation*}
\begin{equation*}
\Longleftrightarrow \quad
\left\{
\begin{aligned}
& \begin{bmatrix}
      \bm{\Lambda} & \dfrac{1}{2} (\bm{v}_k + \bm{w}_k +\bm{q}-\bm{a}_k(\bm{x}))\\
      \dfrac{1}{2}(\bm{v}_k + \bm{w}_k +\bm{q}-\bm{a}_k(\bm{x}))^{T} & t-b_k(\bm{x})-\bar{\bm{\mu}}^{T} (\bm{v}_k + \bm{w}_k) - z_{1k} - z_{2k} \\
      \end{bmatrix} \succeq 0,   \\
     & \sqrt{\bar{\gamma_1}} \|\bm{\Sigma}_0^{1/2}\bm{v}_k \|_2 \leq z_{1k}, \\
     & \sqrt{\bar{\gamma_2}} \|\bm{\Sigma}_0^{1/2}\bm{w}_k \|_2 \leq z_{2k}.
\end{aligned}
\right.
\end{equation*}
%}
Then the result follows.
\end{proof}

As for \cref{thm3,thm4}, the proofs are similar to that of \cref{thm1,thm2}, respectively. So we omit them here.

\section{Determine the Value of \texorpdfstring{$\gamma_1$}{gamma1} and \texorpdfstring{$\gamma_2$}{gamma2}}
\subsection{Boundedness Assumption}
Delage and Ye \cite{ye} show the confidence region for the mean and covariance matrix under the boundedness assumption.
\begin{assumption}\label{boundedness}
There exists an $R \geq 0$ such that
$$P\left((\bm{\xi}-\bm{\mu})^T \bm{\Sigma}^{-1} (\bm{\xi}-\bm{\mu}) \leq R^2 \right) = 1,$$
where $\bm{\mu}$ is the mean of $\bm{\xi}$ and $\bm{\Sigma}$ is the covariance matrix of $\bm{\xi}$.
\end{assumption}

\begin{theorem}[Delage and Ye\cite{ye}]
\label{confidence}
Let $\{\bm{\xi}_i\}_{i=1}^M$ be a set of $M$ samples generated independently and randomly according to the distribution of $\bm{\xi}$. If $\bm{\xi}$ satisfies \cref{boundedness} and $M$ satisfies the following inequality
$$M > R^4(\sqrt{1-p/R^4}+\sqrt{\ln(2/\delta)})^2,$$
then with the probability greater than $1-\delta$ over the choice of $\{\bm{\xi}_i\}_{i=1}^M$, the following set of constraints are met:
\begin{equation*}
\begin{split}
& (\bm{\mu}_0-\bm{\mu})^T\Sigma^{-1}(\bm{\mu}_0-\bm{\mu}) \leq \beta(\delta/2), \\
& \bm{\Sigma} \preceq \dfrac{1}{1-\alpha(\delta/4)-\beta(\delta/2)} \bm{\Sigma}_0, \\
& \bm{\Sigma} \succeq \dfrac{1}{1+\alpha(\delta/4)} \bm{\Sigma}_0,
\end{split}
\end{equation*}
where $\bm{\mu}_0$ is the sample mean, $\bm{\Sigma}_0$ is the sample covariance matrix and $\alpha(\delta/4)$ and $\beta(\delta/2)$ are as follows:
\begin{equation*}
\begin{aligned}
\alpha(\delta/4) & = (R^2 / \sqrt{M})(\sqrt{1-p/R^4} + \sqrt{\ln(4/\delta)},\\
\beta(\delta/2) & = (R^2/M)(2 + \sqrt{2\ln(2/\delta)})^2.
\end{aligned}
\end{equation*}
\end{theorem}

Moreover, they show that if the conditions in \cref{confidence} are satisfied, the ambiguity set \cref{ambiguity} with $\gamma_1$ and $\gamma_2$ defined as follows contains the true distribution of $\bm{\xi}$ with probability greater than $1-\delta$ over the choice of $\{\bm{\xi}_i\}_{i=1}^M$.
\begin{equation}\label{bounded_gamma}
\begin{split}
\gamma_1 = \dfrac{\beta(\delta/2)}{1-\alpha(\delta/4)-\beta(\delta/2)}, \\
\gamma_2 = \dfrac{1+\beta(\delta/2)}{1-\alpha(\delta/4)-\beta(\delta/2)}.
\end{split}
\end{equation}

\subsection{Normal Assumption}
We consider an unbounded case that the true distribution is normal distribution.
\begin{lemma}[see \cite{multi_sta_book} section 8.3]
\label{first_moment}
Let $\{\bm{\xi}_i\}_{i=1}^M$ be i.i.d $N_p(\bm{\mu},\bm{\Sigma})$. By Hotelling-$T^2$ test, we have exactly
$$P\left((\bm{\mu}_0 - \bm{\mu})^T \bm{\Sigma}_0^{-1} (\bm{\mu}_0 - \bm{\mu}) \leq \dfrac{p}{M-p} F_{\alpha}(p,M-p) \right) = 1- \alpha,$$
where $F_{\alpha}(p,M-p)$ is $\alpha$ quantile of the distribution $F(p,M-p)$.
\end{lemma}

\begin{lemma}[see \cite{multi_sta_book} Proposition 7.1]
\label{wishart}
For $\{\bm{\xi}_i\}_{i=1}^M$  i.i.d $N_p(\bm{\mu},\bm{\Sigma})$,
$$M\bm{\Sigma}_0 \sim  W_p(M-1,\bm{\Sigma}),$$
where $W_p(M-1,\bm{\Sigma})$ is Wishart distribution.
\end{lemma}

The joint distribution of eigenvalues and the probability density function of the maximum and minimum eigenvalues of Wishart matrix all have closed form \cite{eigenpdf}. However, it's still difficult to compute the confidence region of the maximum and minimum eigenvalues. Hence, we use the approximation of $P(a \leq \lambda_{\min}(W))$.

\begin{lemma}[see \cite{interval} section 4]\label{approx}
For $\bm{W} \sim W_p(m,\bm{\mathcal{I}})$,
\begin{equation*}
\begin{split}
P(\lambda_{\max}(\bm{W}) < b) \to F_{\beta}(\dfrac{b-\mu_{mp}}{\sigma_{mp}}), \\
%\simeq P\left(k,\dfrac{(\alpha + (b-\mu_{mp})/\sigma_{mp})^+}{\theta}\right) \\
P(\lambda_{\min}(\bm{W}) > a) \to F_{\beta}(-\dfrac{a-\mu^-_{mp}}{\sigma^-_{mp}}),
%\simeq P\left(k,\dfrac{(\alpha - (a-\mu^-_{mp})/\sigma^-_{mp})^+}{\theta}\right)
\end{split}
\end{equation*}
where
\begin{equation*}
\begin{split}
\mu_{mp} = (\sqrt{m}+\sqrt{p})^2,
\sigma_{mp} = \sqrt{\mu_{mp}}(\dfrac{1}{\sqrt{p}}+\dfrac{1}{\sqrt{m}})^{\frac{1}{3}}, \\
\mu^-_{mp} = (\sqrt{m}-\sqrt{p})^2,
\sigma^-_{mp} = \sqrt{\mu^-_{mp}} (\dfrac{1}{\sqrt{p}}-\dfrac{1}{\sqrt{m}})^{\frac{1}{3}}.
\end{split}
\end{equation*}
Besides, $F_{\beta}(\cdot)$ is the distribution function of Tracy-Widom distribution $TW_{\beta}$ and $\beta=1$ for real matrices and $\beta = 2$ for complex matrices.
\end{lemma}

\begin{corollary}
Let $\{\bm{\xi}_i\}_{i=1}^M$ be i.i.d $N_p(\bm{\mu},\bm{\Sigma})$ and
\begin{equation}\label{normal_gamma}
\begin{split}
& \gamma_1 = \frac{p}{M-p} F_{\alpha}(p,M-p), \\
& \gamma_2 = \gamma_1 + \frac{M}{a}.
\end{split}
\end{equation}
Then the constraint
\begin{equation}\label{c1}
(E_P[\bm{\xi}] - \bm{\mu}_0)^T \bm{\Sigma}_0^{-1} (E_P[\bm{\xi}] - \bm{\mu}_0) \leq \gamma_1
\end{equation}
is met with the probability $1-\alpha$. And the constraint
\begin{equation}\label{c2}
E_P[(\bm{\xi-\bm{\mu}_0})(\bm{\xi-\bm{\mu}_0})^T] \preceq \gamma_2 \bm{\Sigma}_0
\end{equation}
is met with the probability $F_{\beta}(-\frac{a-\mu^-_{M-1, p}}{\sigma^-_{M-1, p}})$ approximately.
\end{corollary}

\begin{proof}
For the first part of the corollary, the proof follows from \cref{first_moment} directly. We only need to proof the remainder of the corollary.
According to \cref{wishart},
$$M\bm{\Sigma}_0 \sim  W_p(M-1,\bm{\Sigma}).$$
Hence,
$$\bm{W} = M \bm{\Sigma}^{-\frac{1}{2}} \bm{\Sigma}_0 \bm{\Sigma}^{-\frac{1}{2}} \sim W_p(M-1,\bm{\mathcal{I}}).$$
According to \cref{approx},
$$P(\lambda_{\min}(\bm{W}) > a) \to F_{\beta}(-\dfrac{a-\mu^-_{M-1,p}}{\sigma^-_{M-1,p}}).$$
Hence,
\begin{equation*}
P(\bm{\Sigma} \preceq \dfrac{M}{a}\bm{\Sigma}_0) = P(a \bm{\mathcal{I}} \preceq \bm{W}) = P(\lambda_{\min}(\bm{W}) \geq a)
\approx F_{\beta}(-\dfrac{a-\mu^-_{M-1, p}}{\sigma^-_{M-1, p}}).
\end{equation*}
In addition,
\begin{equation*}
\begin{split}
 \dfrac{M}{a}\bm{\Sigma}_0 \succeq \bm{\Sigma} & = E_P[\bm{\xi}\bm{\xi}^T] - E_P[\bm{\xi}] E_P[\bm{\xi}]^T \\
 & \succeq E_P[(\bm{\xi-\bm{\mu}_0})(\bm{\xi-\bm{\mu}_0})^T] - \gamma_1 \bm{\Sigma}_0.
\end{split}
\end{equation*}
Therefore, \cref{c2} is met with the probability greater than $F_{\beta}(-\frac{a-\mu^-_{M-1, p}}{\sigma^-_{M-1, p}})$ approximately.
\end{proof}

Let $a = \mu^-_{M-1, p} -\sigma^-_{M-1, p}$ and $\beta =1$, then $F_1(1) \approx 95\%$.

\bibliographystyle{siamplain}
\bibliography{ref}
\end{document}

%% file: GDRO_shared.tex
\usepackage{lipsum}
\usepackage{algorithmic}
\usepackage{graphicx,epstopdf}
\usepackage[caption=false]{subfig}
\usepackage{amsmath,amsfonts}
\usepackage{amssymb}
\usepackage{array}
\usepackage{mathrsfs}
\usepackage{enumerate}
\usepackage{geometry}
\usepackage{dsfont}
\usepackage{bbm}
\usepackage{bm}
\usepackage{appendix}
\usepackage{url}
\usepackage{multirow}
\usepackage{booktabs}
\usepackage{threeparttable}
\ifpdf
  \DeclareGraphicsExtensions{.eps,.pdf,.png,.jpg}
\else
  \DeclareGraphicsExtensions{.eps}
\fi

\newsiamremark{remark}{Remark}
\newsiamremark{hypothesis}{Hypothesis}
\crefname{hypothesis}{Hypothesis}{Hypotheses}
\newsiamthm{claim}{Claim}
\newsiamremark{assumption}{Assumption}

\headers{Globalized distributionally robust optimization based on samples}{Yueyao Li, and Wenxun Xing}

\title{Globalized distributionally robust optimization based on samples\thanks{April 29th.
\funding{This work is supported by the National Science Foundation of China Grant \#11771243.}}}

\author{Yueyao Li\thanks{Department of Mathematical Sciences, Tsinghua University, Beijing 100084, People's Republic of China (\email{liyueyao18@mails.tsinghua.edu.cn}).}
\and Wenxun Xing\thanks{Department of Mathematical Sciences, Tsinghua University, Beijing 100084, People's Republic of China (\email{wxing@mail.tsinghua.edu.cn}).}
}

\usepackage{amsopn}